\documentclass[12pt]{amsart}
\usepackage[utf8]{inputenc}
\usepackage[initials]{amsrefs}
\usepackage[
        colorlinks,
]{hyperref}

\usepackage{fullpage,amssymb,amsmath,dsfont, xcolor}

\numberwithin{equation}{section}

\newcommand{\Gal}{{\rm Gal}}
\newcommand{\Ord}{{\rm ord}}

\newcommand{\GL}{{\rm GL}}

\renewcommand{\d}{{\rm d}}

\newcommand{\Q}{\mathbb Q}
\newcommand{\R}{\mathbb R}

\newcommand{\C}{\mathbb C}
\newcommand{\Z}{\mathbb Z}
\newcommand{\Irr}{{\rm Irr}}

\newcommand{\ba}{\mathbf{a}}

\newcommand{\bx}{\mathbf{x}}
\newcommand{\by}{\mathbf{y}}
\newcommand{\bz}{\mathbf{z}}
\newcommand{\bt}{\mathbf{t}}
\newcommand{\bs}{\mathbf{s}}
\newcommand{\bE}{\mathbf{E}}
\newcommand{\bB}{\mathbf{B}}
\newcommand{\bS}{\mathbf{S}}
\newcommand{\bT}{\mathbf{T}}
\newcommand{\bX}{\mathbf{X}}

\newcommand{\J}{\text{J}_0}

\newcommand{\ve}{\varepsilon}
\newcommand{\vp}{\varphi}

\renewcommand{\Re}{{\mathfrak R}{\rm e}}

\newtheorem{X}{X}[section]
\newtheorem{cor}[X]{Corollary}
\newtheorem{lem}[X]{Lemma}
\newtheorem{prop}[X]{Proposition}
\newtheorem{thm}[X]{Theorem}

\theoremstyle{definition}
\newtheorem{defi}[X]{Definition}

\usepackage{amsthm}

\theoremstyle{plain}
\newtheorem*{thma}{Theorem A}
\newtheorem*{thmb}{Theorem B}
\newtheorem*{thmc}{Theorem C}
\newtheorem*{thmd}{Theorem D}

\theoremstyle{remark}
\newtheorem{rk}{Remark}[section]
\newtheorem*{exe}{Example}
\newtheorem*{examples}{Examples}

\begin{document}
\title{Prime ideal races with several competitors}
\author{A. Bailleul}
\address{ENS Paris-Saclay, Centre Borelli, UMR 9010, 91190 Gif-sur-Yvette, France}
\email{alexandre.bailleul@ens-paris-saclay.fr}
\author{M. Hayani}
\address{IMB, 351 cours de la Libération, 33400 Talence, France}
\email{mounir.hayani@math.u-bordeaux.fr}
\date{}
\maketitle

\begin{abstract}
We investigate races among prime ideals in number fields when there are two or more competing conjugacy classes. In their work \cite{FJ}, Fiorilli and Jouve studied two‐way races in number fields and showed that—unlike the classical setting of primes in arithmetic progressions—these biases can approach the extreme values of $0$ and $1$. They also identified when these biases tend toward one‐half (as the degree of the extension grows), which we call “moderate biases” because that behavior mirrors the classical case.

In this paper, we extend their analysis to races with $r$ competing conjugacy classes (\emph{r}-way races) and precisely study the cases where these biases are moderate (meaning they tend to $1/r!$ as the discriminant of the extension grows). Our first main result is an explicit formula for the bias in any $r$-way race (for all $r \ge 2$), generalizing the two‐way formula of Fiorilli and Jouve~\cite{FJ} and Lamzouri’s $r$-way expression in the classical case of residue classes modulo $q$~\cite{Lam}, under the same hypotheses. Using this formula, we give a criterion characterizing completely, in the abelian case, $r$-moderate races. Surprisingly, once $r\ge 3$ this criterion is independent of $r$, making the two-way race exceptional. We also construct families of number fields exhibiting such moderacy, we prove density results for the values of logarithmic densities and exhibit different behaviors of those densities between the cases $r=3$ and $r \ge 4$.
\end{abstract}

\section{Introduction}\label{sec:intro}
In 1853, Chebyshev observed that primes congruent to $3\pmod 4$ seem to appear more often than those congruent to $1\pmod 4$. This phenomenon is now known as ``Chebyshev's bias''. Let $q\ge 3$ and $a \in \Z$ coprime to $q$. Define 
\[\pi(x; q,a):=\#\{\, p\le x\, \colon\, p\equiv a\pmod q \, \} \, . \]
The prime number theorem for arithmetic progressions guarantees that for any $a,b\in \Z$ coprime to $q$ one has $\pi(x ; q,a)/\pi(x;q,b) \longrightarrow 1$ as $x\to \infty$. Yet it gives no control over the difference $\pi(x;q,a)-\pi(x;q,b)$. In 1914, Littlewood~\cite{Littlewood} showed that $\pi(x;4,3)-\pi(x;4,1)$ changes sign infinitely often. This problem generalizes as follows: for an integer \(q\ge3\) and an integer \(2\le r\le\varphi(q)\), let \( \mathcal{A}_r(q)\) be the set of $r$-tuples of pairwise distinct invertible residue classes modulo $q$. 
One then asks whether all \(r!\) orderings of the prime counting functions
\(\pi(x;q,a_1),\dots,\pi(x;q,a_r)\) occur infinitely often. In this form, the question is known as the ``Shanks–Rényi prime number race problem.''\\
Under the Generalized Riemann Hypothesis and the Linear Independence hypothesis on the zeros of Dirichlet \(L\)-functions, Rubinstein and Sarnak \cite{RS94} showed that for each \((a_1,\dots,a_r)\in\mathcal{A}_r(q)\) the set
\(
 \mathcal{P}_q(a_1,\dots,a_r):= \bigl\{\,x\ge2 : \pi(x;q,a_1)>\pi(x;q,a_2)>\cdots>\pi(x;q,a_r)\bigr\}
\)
admits a well-defined positive logarithmic density
\[\delta_q^{(r)}(a_1,\dots,a_r):=\lim_{X\to \infty} \frac{1}{\log X}\int_2^X \mathds{1}_{\mathcal{P}_q (a_1,\dots,a_r)}(t) \frac{\mathrm{d}t}{t} \, ,\]
arising from a limiting distribution \(\mu_{q;a_1,\dots,a_r}\). Under those hypotheses, they also proved that
\begin{equation*}\label{classical-r-moderacy}
    \lim_{q\to \infty}\max_{(a_1,\dots,a_r)\in \mathcal{A}_r(q)}\bigl|\delta_q^{(r)}(a_1,\dots,a_r)-1/r!\bigr|=0\,
\end{equation*}
so that asymptotically as $q$ goes to infinity, the biases tend to disappear.
 Fiorilli--Martin~\cite{FiM} proved that the rate of the latter convergence when $r=2$ is of order $1/q^{1/2+o(1)}$. Lamzouri~\cite{Lam} generalized this study to $r\ge 3$, and proved that  the rate of convergence when $r\ge 3$ is much slower and is of order $1/\log q$.\\
In his PhD thesis, Ng \cite{Ng} extended the Rubinstein–Sarnak framework to number field prime ideal races.  Let \(L/K\) be a Galois extension with Galois group \(G = \Gal(L/K)\).  For any class function 
\(
t \colon G \longrightarrow \mathbb{C}
\)
(i.e.\ constant on conjugacy classes), define the prime ideal counting function, for $x\ge 2$
\begin{equation} \label{pi}
\pi(x;L/K;t)
\;:=\;
\sum_{\substack{\mathfrak{p}\,\lhd\, \mathcal{O}_K \\ N(\mathfrak{p}) \,\le x}}
t(\varphi_{\mathfrak{p}}),
\end{equation}
where the sum is over non-zero prime ideals $\mathfrak{p}$ of $\mathcal{O}_K$, and $\varphi_{\mathfrak{p}}$ denotes the corresponding Frobenius conjugacy class (defined up to inertia). 
In particular, if \(C\subset G\) is a conjugacy class, we denote 
\(\pi(x;L/K,C)=\pi(x;L/K,\mathds{1}_C)\).  
Given conjugacy classes \(C_1,\dots,C_r\) in \(G\), Ng studied the set
\[
  \mathcal{P}_{L/K}(C_1,\dots,C_r):=\left\{\,x\ge2 : 
    \frac{\pi(x;L/K,C_1)}{| C_1|}
    < \cdots <
    \frac{\pi(x;L/K,C_r)}{|C_r|}
  \right\}\, .
\]
Under $\mathsf{GRH}$, the Artin conjecture and the linear independence hypothesis over zeros of Artin $L$-functions, Ng proved that the above set admits a positive logarithmic density
\(\delta^{(r)}_{L/K}(C_1,\dots,C_r)\),
arising from a probability measure \(\mu_{C_1,\dots,C_r}\).

Several authors have since turned their attention to the logarithmic densities
\(\delta^{(r)}_{L/K}(C_1,\dots,C_r)\).
A particularly influential contribution is due to Fiorilli and Jouve~\cite{FJ},
who analyzed the two-way races (\(r=2\)).
For conjugacy classes \(C_1,C_2 \subset G=\operatorname{Gal}(L/K)\),
let \(E(C_1,C_2)\) and \(V(C_1,C_2)\) denote, respectively, the mean and variance of the
limiting probability measure \(\mu_{C_1,C_2}\).
They introduced the important parameter
\[
B_{L/K}(C_1,C_2)\;:=\;
\frac{E(C_1,C_2)}{\sqrt{V(C_1,C_2)}},
\]
and showed that its magnitude governs the bias:
\begin{itemize}
  \item \textbf{Extreme bias.} If
        \(B_{L/K}(C_1,C_2) \to \pm \infty\) as the degree $[L:\Q]\to \infty$,
        then the density \(\delta^{(2)}_{L/K}(C_1,C_2)\)
        converges to the extreme values \(1\) or \(0\).
  \item \textbf{Moderate bias.} Under additional hypotheses
        (automatically satisfied in the abelian case), if
        \(B_{L/K}(C_1,C_2)\to 0\) as $[L:K]\to \infty $,
        then \(\delta^{(2)}_{L/K}(C_1,C_2)\to 1/2\),
        mirroring the classical prime-ideal race studied by
        Rubinstein-Sarnak and Fiorilli-Martin.
\end{itemize}
Applying this criterion to several explicit families such that the degree grows to $\infty$, Fiorilli and Jouve produced examples of both moderate and extreme biases along families of number fields.
In the same spirit, Bailleul~\cite{Bail} investigated dihedral and quaternion
extensions, uncovering further families in which the densities not only
approach the value \(1/2\) but can again be made to lie
arbitrarily near the extremes, while also highlighting the influence of central zeros of Artin L-functions on the behavior of those densities.\\
In this paper, we will be focusing on the case of Galois extensions $L/\Q$ over $\Q$, instead of general Galois extensions of number fields $L/K$. When $K/\Q$ is Galois, the study of $\pi(x; L/K; t)$ as defined in \ref{pi} reduces to that of $\pi(x; L/\Q, t^+)$, where $t^+$ is the induced class function on $\Gal(L/\Q)$ in the sense of character theory (see for instance~\cite{hayani}*{Corollary 2.3}). On the other hand, when $K/\Q$ is not Galois, complications may arise because distinct class functions on $\Gal(L/K)$ may induce the same class function on $\Gal(L/\Q)$, as has been illustrated in~\cites{FJ2, hayani}. We will not be considering these kinds of situations here.

When $L/\Q$ is a Galois extension with Galois group $G$, define $\mathcal{A}_r(G)$ to be the set of $r$-tuples of pairwise distinct conjugacy classes in $G$. When $G$ is abelian, we identify $\mathcal{A}_r(G)$ with the set of $r$-tuples of elements of $G$ by abuse of notation since conjugacy classes are singletons in this case. Similarly, in the abelian case one can define $\delta^{(r)}_{L/\Q}(a_1,\dots,a_r)$ as the value of $\delta^{(r)}_{L/\Q}$ at the corresponding conjugacy classes (and the same goes for all functions previously defined on conjugacy classes).

Our first main result (Theorem~\ref{explicit-formula}) is an explicit formula, valid for all finite Galois extensions $L/\Q$, relating $\delta_{L/\Q}^{(r)}(C_1,\dots,C_r)$ with the cumulative distribution function of a centered Gaussian vector whose covariance matrix is related to zeros of Artin $L$-functions of the extension $L/\Q$. This explicit formula will be stated in the following paragraph due to the technical details needed to state it.

The following notion will be one of the central point of focus of this paper.

\begin{defi}
    Let $(L_n)_n$ be a family of Galois extensions over $\Q$ with respective Galois groups $(G_n)_n$, such that $[L_n:\Q] \longrightarrow \infty$ as $n\to \infty$, and for which $\mathsf{GRH}$, $\mathsf{AC}$, and $\mathsf{LI}$ hold (see~\S~\ref{subsec:hypotheses} for precise statements). When $r\ge 2$, we say that $(L_n)_n$ is uniformly $r$-moderate if 
    \[ \lim_{n\to \infty}\max_{(C_1,\dots,C_r)\in \mathcal{A}_r(G_n)}\left|\delta_{L_n/\Q}^{(r)}(C_1,\dots,C_r)-\frac{1}{r!}\right| =0\, .\]
\end{defi}

When $G$ is a finite abelian group, define \(r(G):=\#\{\, x\in G\, :\, x^2=1\, \}\, \) the number of square roots of $1$ in $G$. 
Our main contribution for two-way races (two conjugacy classes) is that in the abelian case we are able to give a simple criterion for the logarithmic densities $\delta_{L/\Q}^{(2)}(C_1,C_2)$ to approach any value $\alpha \in [0,1]$: 
\begin{thma}\label{thm : two-way-races}
    Let $\ell\in [0,\infty]$ and let $(L_n)_n$ be a family of abelian Galois extensions over $\Q$ with respective Galois groups $(G_n)_n$, such that $d_{L_n/\Q}\to \infty$ as $n\to \infty$, and for which $\mathsf{GRH}$ and $\mathsf{LI}$ hold. Assume that 
    \[\lim_{n\to \infty} \frac{r(G_n)}{\sqrt{\log d_{L_n}}}=\ell\, ,\]
    where $d_{L_n}$ is the absolute discriminant of $L_n$, then: \begin{itemize}
        \item if $\ell= 0$, then $(L_n)_n$ is uniformly $2$-moderate.
        \item if $\ell\ne 0$, then there exist sequences $(a_n)_n,(b_n)_n$ such that for all $n\ge 1$, we have $a_n,b_n\in G_n$ and
        \[\lim_{n\to \infty} \delta_{L_{n}/\Q}^{(2)}(a_n,b_n)=\frac{1}{\sqrt{2\pi}}\int_{-\infty}^{\ell/\sqrt{2}}\exp\left(-\frac{x^2}{2}\right)\, \mathrm{d}x\, .\]
    \end{itemize}
    Moreover, if $L_n \subset L_{n+1}$ for all $n\ge 1$ then the sequences $(a_n)_n$ and $(b_n)_n$ can be expressed explicitly.
\end{thma}
In particular, in the abelian case, Theorem~A characterizes uniformly $2$-moderate families: they are exactly families $(L_n)_n$, satisfying the assumption of Theorem~A, and such that $\lim_{n \to \infty} r(G_n)/(\log d_{L_n})^{1/2}=0$. Theorem~\ref{thm : uniform-moderacy} gives a full characterization of $r$-moderate families for $r\ge 3$, which demands more information about the ramification data of the family.  In Definition~\ref{def:pointwise} we introduce pointwise moderacy, a weaker form of moderacy for the "prime ideal races" we consider. Our second main result shows that $r=2$ behaves differently from all $r\ge 3$ in the uniform setting, yet this dichotomy essentially disappears for pointwise moderacy.
\begin{thmb}\label{thm:B}
    Let $(L_n)_n$ be a family of abelian Galois extensions over $\Q$ with respective Galois groups $(G_n)_n$, such that $d_{L_n/\Q}\to \infty$ as $n\to \infty$, and for which $\mathsf{GRH}$ and $\mathsf{LI}$ hold. Then: 
    \begin{enumerate}
        \item for $r\ge 3$, $(L_n)_n$ is uniformly $r$-moderate if and only if $(L_n)_n$ is uniformly $3$-moderate.
        \item if moreover $(L_n)_n$ is increasing, then for $r\ge2$, $(L_n)_n$ is pointwise $r$-moderate if and only if it is uniformly $2$-moderate.
    \end{enumerate}
\end{thmb}
A consequence of Theorem~\ref{thm : uniform-moderacy}, is that pointwise $r$-moderate extensions are automatically $2$-moderate (for $r\ge 3$). The converse is not true, as one can construct $2$-moderate extensions such that, for any $r \ge 3$, the values of $\delta_{L_n/\Q}^{(r)}$ are dense in a set with non-empty interior:
\begin{thmc}\label{thm : 2mod-vs-rdense}
    There exists an increasing sequence of prime numbers $(p_n)_n$ such that if $L_n=\Q\bigl(\sqrt{p_1},\dots,\sqrt{p_n}\bigr)$ satisfies $\mathsf{GRH}$ and $\mathsf{LI}$ for all $n\ge 1$, then, denoting $G_n=\Gal(L_n/\Q)$, we have that $(L_n)_n$ is uniformly $2$-moderate and for all $r\ge 3$ the set 
    \[ \overline{\bigcup_{n\ge1}\delta_{L_n/\Q}^{(r)}\bigl(\mathcal{A}_r(G_n)\bigr)}\]
    has non-empty interior.
\end{thmc}
We also prove (in \S~\ref{sec:density}) that one can construct explicit
families $(L_n)_n$ with
\[
  \overline{\bigcup_{n\ge1}\delta_{L_n/\Q}^{(2)}
     \bigl(\mathcal A_2(G_n)\bigr)}=[0,1].
\]
Thus the logarithmic densities related to Chebyshev's bias in number fields can be arbitrarily close to any value in $[0,1]$.\\
Fiorilli--Jouve's two-way criterion (and Corollary~\ref{r=1}) shows that 
$B_{L/\Q}(C_1,C_2)\to\pm\infty$ if and only if
$\delta_{L/\Q}^{(2)}(C_1,C_2)\to1$ or $0$.
The picture changes for $r$-way races when $r\ge3$. If two-way races are moderate, then a $3$-way race is never extreme; nevertheless, one can still produce extreme $r$-way races for $r\ge 4$ --revealing a sharp split between the cases $r=3$ and $r \geq 4$:


\begin{thmd}\label{thm:D}
Let $(L_n)_n$ be a $2$-moderate family of abelian extensions of $\Q$, for which $\mathsf{GRH}$ and $\mathsf{LI}$ hold, with respective Galois groups $(G_n)_n$, and such that $[L_n : \Q]\to \infty$. Then for $r=3$ one has $$\delta^{(3)}_{L_n/\Q}(\mathcal{A}_3(G_n)) \subset \left[\frac{1}{4} - \frac{1}{2\pi} \arcsin\left(\frac{3}{4}\right) + o(1), \frac{1}{4} + o(1)\right].$$ However for $r \ge 4$, there exists a family $(K_n)_n$ of abelian extensions over $\Q$ with Galois groups $(H_n)_n$ such that $$0\in \overline{\bigcup_{n\geq 1} \delta_{K_n/\Q}^{(r)}(\mathcal{A}_r(H_n))}\, .$$
\end{thmd}
Theorems~C and~D follow directly from our main density theorem (Theorem~\ref{main-density}).  
Its proof hinges on two complementary ideas:  
\begin{enumerate}
\item \emph{Arithmetic selection.}  
      For each extension \(L_n/\Q\), we construct carefully chosen \(r\)-tuples of conjugacy classes whose covariance matrices
      \(\Delta_{L_n/\Q}\) converge to a chosen positive-definite limit.  
\item \emph{Probabilistic comparison.}  
      Using a Gaussian comparison principle (notably Slepian’s Lemma), we translate the limiting behavior of
      \(\Delta_{L_n/\Q}\) into quantitative control of the race densities \(\delta_{L_n/\Q}^{(r)}\).  
\end{enumerate}
This arithmetic–probabilistic strategy shows that the set of attainable values of \(\delta_{L_n/\Q}^{(r)}\) can, in the limit,
fill intervals, leading both to the non-empty interior phenomenon of Theorem~C and to the extreme values established in Theorem~D.

Let us briefly outline the structure of this paper. In
\S~\ref{sec:statements} we fix the analytic assumptions ($\mathsf{GRH}$, the Artin conjecture and $\mathsf{LI}$). We also state the main theorems we will be proving throughout this paper. In \S~\ref{sec:tools} we assemble the arithmetic background on Artin conductors and undertake a
detailed study of several towers of multiquadratic extensions, which will help us provide both
examples and counter-examples related to our main results.
In \S~\ref{sec:explicitformula} we establish the explicit formula for prime ideal race densities;
the derivation hinges on a recent multidimensional Berry–Esseen inequality and
requires substantial additional combinatorial work to adapt it to Artin~$L$-functions.
In \S~\ref{sec:pointwise} and \S~\ref{sec:uniform} we develop a series of lemmas that give quantitative control of the
covariance matrix attached to an $r$-tuple of class functions—bounds that are
crucial for applying the explicit formula.  In \S~\ref{sec:pointwise} we settle the two-way case by proving Theorem~A and
prove the pointwise part of Theorem~B.  In \S~\ref{sec:uniform} we prove Theorem~\ref{thm : uniform-moderacy} which finishes the proof of Theorem~B.
In \S~\ref{sec:density} we combine the covariance lemmas of \S~\ref{sec:pointwise} and \S~\ref{sec:uniform} with some probabilistic results (such as Slepian's Lemma) to prove our main density result (Theorem~\ref{main-density}). This result, together with the multiquadratic examples, yield Theorems~C and~D.

\section*{Notation}

For the convenience of the reader, we collect the notation used throughout the paper. Each item will be introduced when it is first needed. Here, $L/\Q$ is a finite Galois extension, $G$ is its Galois group, $G^{\sharp}$ is the set of conjugacy classes of $G$, the $t_i$ are real-valued class functions on $G$ satisfying $\langle t_i, 1\rangle_G=0$, and we set $\bt = (t_1, \dots, t_r)$. When $C_1, C_2 \in G^{\sharp}$, we let $t_{C_1,C_2}=\frac{|G|}{|C_1|}\mathds{1}_{C_1}-\frac{|G|}{|C_2|}\mathds{1}_{C_2}$. A summation over $\chi$ is a summation over irreducible characters of $G$, and one over $\gamma_{\chi}$ is one over the imaginary parts of non-trivial zeros of $L(s, L/\Q, \chi)$.

\begin{itemize}
    \item For $f,g:G\to \C$ define the scalar product of $f$ and $g$ by: 
\(
\langle f,g \rangle_G:=\frac{1}{|G|}\sum_{x\in G}f(x)\overline{g(x)}
\).
When the functions in question are clearly defined on $G$, we simply denote $\langle f,g \rangle:=\langle f,g\rangle_G$.
    \item $\mathcal{A}_r(G) := \left\{(C_1,\dots,C_r)\in(G^{\sharp})^r\, :\, C_i\ne C_j\text{ for }i\ne j\right\}$ is the set of $r$-tuples of conjugacy classes of $G$.
    \item \( \| \widehat{\bt} \|_\infty:=\max_{1\le i\le r}\max_{\chi \in \Irr(G)} |\langle t_i,\chi\rangle| \, .\)
    \item $r(G) := \#\{x \in G : x^2 = 1\}$ is the number of square roots of $1$ in $G$.
    \item $r_G(g):=\# \{ x \in G\, :\, x^2=g \}$ is the number of square roots of $g$ in $G$.
    \item $d_L$ is the absolute value of the discriminant of $L/\Q$.
    \item When $L/K$ is a Galois extension of number fields and $\chi$ is a complex character of $\Gal(L/K)$, $A(\chi):=d_K^{\chi(1)}N_{K/\Q}\left(\mathfrak{f}(L/K,\chi) \right)$ is a quantity related to the Artin conductor $\mathfrak{f}(L/K, \chi)$ of $\chi$.
    \item $E(t):=-\langle t,r_G\rangle-\sum_{\chi \ne 1} \langle t,\chi\rangle \Ord_{s=1/2}L(s,L/\Q,\chi)$ is the mean value of the limiting distribution $\mu_t$.
    \item $V(t):=2\sum_{\chi \ne1}|\langle t,\chi\rangle|^2\sum_{\gamma_\chi>0}\frac{1}{\frac{1}{4}+\gamma_\chi^2}$ is the variance of the limiting distribution $\mu_t$.
    \item $B(t):= \frac{E(t)}{\sqrt{V(t)}}$ is the inverse of the coefficient of variation of the limiting distribution $\mu_t$.
    \item $\rho(t_i,t_j):=\frac{1}{\sqrt{V(t_i)V(t_j)}} \sum_{\chi \ne 1}\sum_{\gamma_\chi>0}\frac{2 \Re\left( \langle t_i,\chi\rangle \overline{\langle t_j,\chi\rangle}\right)}{\frac{1}{4}+\gamma_\chi^2}$ is the $(i, j)$ entry of the covariance matrix associated to $\bt$.
    \item $\Delta(\bt):=\bigl(\rho(t_i,t_j)\bigr)_{1\le i,j\le r}$ is the covariance matrix associated to $\bt$.
    \item $\lambda_\bt$ is the minimal eigenvalue of the matrix $\Delta(\bt)$.
    \item $N_{L}:=2\sum_{\chi \ne 1}\sum_{\gamma_\chi>0}\frac{1}{\frac{1}{4}+\gamma_\chi^2}$ is a quantity that arises in variance estimates.
    \item The quantities \begin{align*}U_{L/\Q}\colon\;G\setminus \{1\}&\longrightarrow \R\\ \nonumber
\quad a &\longmapsto \frac{1}{N_{L}} \sum_{\chi \ne 1}\sum_{\gamma_\chi>0}\frac{2 \Re(\chi(a))}{\frac{1}{4}+\gamma_\chi^2}\end{align*} and \begin{align*}S_{L/\Q}\colon\;\bigl(G\setminus \{1\}\bigr)^2&\longrightarrow \R\\ \nonumber
\quad (a,b) &\longmapsto U_{L/\Q}(a)-U_{L/\Q}(b)\end{align*} will allow us to characterize $r$-moderate abelian extensions.
    \item $T_{L/\Q}(a,b) := \frac{1}{N_{L}}V(t_{\{a\}, \{b\}})$ is the normalized variance in the abelian case.
    \item $\Gamma_r$ is the tridiagonal matrix with $1$ on the main diagonal and $- \frac{1}{2}$ on the upper and lower diagonal, which is the limiting covariance matrix for $r+1$-moderate extensions (see \eqref{Gamma-r}).
    \item $\bx \mapsto F_r(\bx;\Gamma)$ is the distribution function of a centered Gaussian random vector with covariance matrix $\Gamma$.
\end{itemize}

We use the following notation: if $f$ is a function defined on vectors of $r$ class functions $(t_1,\dots,t_r)$ of $G=\Gal(L/\Q)$, we define an associated function $f_{L/\Q}^{(r)}:\mathcal{A}_{r+1}(G)\longrightarrow\mathcal{V}$ by 
\begin{equation}\label{f(t)tof(C)}
    f_{L/\Q}^{(r)}\bigl(C_1,\dots,C_{r+1}\bigr)=f\bigl(t_{C_1,C_2},\dots,t_{C_r,C_{r+1}}\bigr)\, .
\end{equation}
In particular, the previously defined functions $\Delta,\ V,\ E,\ B$ induce functions $\Delta_{L/\Q}^{(r)}:\mathcal{A}_{r+1}(G)\to \mathcal{M}_r(\R)$, $V_{L/\Q}:=V_{L/\Q}^{(1)}\, \colon\, \mathcal{A}_2(G)\to \R$, etc.

\section{Framework and statements of main theorems}\label{sec:statements}
\subsection{Analytic hypotheses and notation}\label{subsec:hypotheses}
Let $L/\Q$ be a finite Galois extension with group $G$, following the notation of~\cite{FJ}, the hypotheses that we will assume are the following: 
\begin{itemize}
  \item $\mathsf{GRH}$: We assume the Generalized Riemann Hypothesis for $L/\Q$, that is, for every $\chi\in\Irr(G)$, every non-trivial zero of $L(s,L/\Q,\chi)$ lies on the critical line $\{\Re(s)=\frac{1}{2}\}$.
  \item $\mathsf{AC}$: We will assume that for every non-trivial $\chi\in\Irr(G)$, the Artin $L$-function $L(s,L/\Q,\chi)$ is entire.
  \item $\mathsf{LI}^-$: We assume that the multiset
  \[
  \Bigl\{\, \gamma>0 \colon \exists\, \chi\in\Irr(G)\setminus\{1\}\text{ with } L\Bigl(\tfrac12 + i\gamma,L/\Q,\chi\Bigr)=0 \Bigr\}
  \]
  is linearly independent over $\Q$.
  \item $\mathsf{LI}$: When $L/\Q$ is abelian, we assume $\mathsf{LI}^-$ and $L(\tfrac{1}{2},L/\Q,\chi)\ne 0$ if $\chi \in \Irr(G)$.
\end{itemize}

Note that $\mathsf{AC}$ holds for abelian extensions of $\Q$ since class field theory tells us that Artin $L$-functions are Hecke $L$-functions in that case. As for $\mathsf{LI}^-$, we expect it to be true as it is known that those $L$-functions do not satisfy any algebraic differential equation except for their functional equation (as was proven by Ostrowski \cite{Ost}). This was introduced in Ng's thesis~\cite{Ng}. Finally, the hypothesis $\mathsf{LI}$ was introduced by Fiorilli--Jouve~\cite{FJ}*{Section 1.2} in a greater generality and it allows to control the contribution of central zeros to the behavior of the counting functions $\pi(x,L/\Q, t)$, as was done in \cite{Bail} for example. In the case where $L/\Q$ is abelian, then by the Kronecker-Weber theorem, $L$ is included in a cyclotomic extension of $\Q$, and the functorial properties of Artin $L$-functions imply that those $L$-functions are factors of Dirichlet $L$-functions. Therefore, in that context the non-vanishing assumption in $\mathsf{LI}$ is equivalent to Chowla's conjecture on the non-vanishing of Dirichlet $L$-functions at $1/2$, and our hypothesis $\mathsf{LI}$ is the same as Lamzouri's in \cite{Lam}.

For $f,g:G\to \C$ define the scalar product of $f$ and $g$ by: 
\[
\langle f,g \rangle_G:=\frac{1}{|G|}\sum_{x\in G}f(x)\overline{g(x)}
\]
When the functions in question are clearly defined on $G$, we simply denote $\langle f,g \rangle:=\langle f,g\rangle_G$.
Let $r\ge 1$, and let $t_1,\dots,t_r:G\to \R$ be class functions satisfying $\langle t_i,1\rangle=0$ for $1\le i \le r$. Denote $\bt=(t_1,\dots,t_r)$ and set
\begin{equation}\label{setPt}
\mathcal{P}_\bt:=\{ x\geq2\ :\ \pi(x;L/\Q;t_i)<0\text{ for }1\leq i \leq r \}\, .
\end{equation}
We define the logarithmic density of $\mathcal{P}_\bt$ (when it exists) by 
$$\delta(\mathcal{P}_\bt) := \lim_{X \to \infty} \frac{1}{\log X} \int_2^X \mathds{1}_{\mathcal{P}_\bt}(t) \frac{\d t}{t}\, .$$
When $C_1,C_{2}$ are distinct conjugacy classes in $G$, define the class function $t_{C_1,C_2}$ by 
\begin{equation}\label{tC1C2}
    t_{C_1,C_2}=\frac{|G|}{|C_1|}\mathds{1}_{C_1}-\frac{|G|}{|C_2|}\mathds{1}_{C_2} \, .
\end{equation}
With this notation, if $C_1,\dots,C_{r+1}$ are pairwise distinct conjugacy classes, and if $\bt'=\bigl(t_{C_1,C_2},\dots,t_{C_r,C_{r+1}}\bigr)$ we have
\[\mathcal{P}_{\bt'}=\mathcal{P}_{L/\Q}(C_1,\dots,C_r)\quad\text{and}\quad \delta_{\bt'}=\delta_{L/\Q}^{(r+1)}(C_1,\dots,C_{r+1})\, .\] 
From now on we assume that the family $\bt=(t_1,\dots,t_r)$ is linearly independent over $\R$. This is to avoid looking at degenerate cases such as $\delta_{L/\Q}^{(3)}(C_1, C_2, C_1)$. Define the function
$g_\bt(x) := (g_{t_1}(x), \dots, g_{t_r}(x))$ where $$g_{t_i}(x) = \frac{\pi(x; L/\Q, t_i)}{x^{1/2}/\log x}\,$$ is the normalized counting function (under $\mathsf{GRH}$).
Define $\bE=(E(t_1),\dots, E(t_r))$ where, denoting $r_G(g)$ the number of $x\in G$ such that $x^2=g$, we have
\begin{equation}\label{mean}
    E(t_i)=-\langle t_i,r_G\rangle-\sum_{\chi \ne 1} \langle t_i,\chi\rangle \Ord_{s=1/2}L(s,L/\Q,\chi)\, .
\end{equation}
Following the work of Rubinstein and Sarnak~\cite{RS94} and Ng~\cite{Ng}*{Chapter 5}, one can prove under the hypotheses $\mathsf{GRH}$, $\mathsf{AC}$, and $\mathsf{LI}^-$, that the function $g_\bt$ admits a limiting distribution $\mu_\bt$ whose Fourier transform is given by:

\begin{equation}\label{fourier-transform}
    \widehat{\mu}_\bt (x_1,\dots,x_r)=e^{-i\langle \bE,\bx\rangle}\prod_{\chi\ne 1}\prod_{\gamma_\chi>0}\J\left(\frac{2|\langle \widehat{\bt}(\chi),\bx \rangle|}{\sqrt{\frac{1}{4}+\gamma_\chi^2}}\right)\,
\end{equation}
where $\bx=(x_1,\dots,x_r)$, $\widehat{\bt}(\chi):=(\langle t_1,\chi\rangle,\dots,\langle t_r,\chi\rangle )\in \C^r\, ,$
$\langle\bx,\by\rangle$ (for $\bx,\by\in \C^r$) is the usual inner product, $J_0$ is the Bessel function of the first kind, $\chi$ varies over the set of non-trivial irreducible characters of $G$, and $\gamma_{\chi}$ varies in the set of positive imaginary parts of zeros of $L(s, L/\Q, \chi)$.\\
For each $i$, define $\mu_{t_i}$ to be the limiting distribution of $g_{t_i}$. Fiorilli--Jouve~\cite{FJ}*{Proposition 3.18} proved that $E(t_i)$ is in fact the mean of $\mu_{t_i}$ and that its variance is given by
\begin{equation}\label{variance}
V(t_i)=2\sum_{\chi \ne1}|\langle t_i,\chi\rangle|^2\sum_{\gamma_\chi>0}\frac{1}{\frac{1}{4}+\gamma_\chi^2}\, .
\end{equation}
\subsection{Explicit formula and main consequences}\label{subsec:statements}
Let $L/\Q$ be a finite Galois extension with Galois group $G$. For all class functions $t_1,t_2\,\colon G\longrightarrow \R$, define the correlation factor \begin{equation}\label{def : rho}
    \rho(t_1,t_2)=\frac{1}{\sqrt{V(t_1)V(t_2)}} \sum_{\chi \ne 1}\sum_{\gamma_\chi>0}\frac{2 \Re\left( \langle t_1,\chi\rangle \overline{\langle t_2,\chi\rangle}\right)}{\frac{1}{4}+\gamma_\chi^2}\, .
\end{equation}
Let $\bt=(t_1,\dots,t_r)$ be a vector of class functions, define
\[ \| \widehat{\bt} \|_\infty=\max_{1\le i\le r}\max_{\chi \in \Irr(G)} |\langle t_i,\chi\rangle| \, .\]
We define the symmetric matrix $\Delta(\bt)=\bigl(\rho(t_i,t_j)\bigr)_{1\le i,j\le r}$. The matrix $\Delta(\bt)$ will play a major role in our study: it is the covariance matrix of the Gaussian random vector used to approximate our prime ideal race densities. In Lemma~\ref{lem : Delta(t)-pos-def}, we prove that if the family $(t_1,\dots,t_r)$ is linearly independent over $\R$, then the symmetric matrix $\Delta(\bt)$ is positive-definite. This shows that the function $\vp_\bt$, defined by 
\begin{equation}\label{def : phibt}
    \vp_\bt(x_1,\dots,x_r):=\exp \left(  -\frac{x_1^2+\dots+x_r^2}{2}-\sum_{1\leq i<j \leq r} \rho(t_i,t_j)x_ix_j \right) \, ,
\end{equation}
is the characteristic function of a non-degenerate centered Gaussian vector with density 
\begin{equation}\label{def : fbt}
    f_\bt (x_1,\dots,x_r)= \frac{1}{((2\pi)^{r}\det \Delta(\bt))^{1/2}}\exp\left(-\frac{1}{2} \bx^T \Delta(\bt)^{-1} \bx \right)\, .
\end{equation}
Note that when $(C_1,\dots,C_{r+1})$ are pairwise distinct conjugacy classes, then the family $\bigl(t_{C_1,C_2},\dots,t_{C_r,C_{r+1}}\bigr)$ is indeed linearly independent over $\R$.
We can now state the explicit formula: 
\begin{thm}[Explicit formula]\label{explicit-formula} Let $L/\Q$ be a Galois extension with group $G$, for which $\mathsf{GRH}$, $\mathsf{AC}$, and $\mathsf{LI}^-$ hold. Let $\bt=(t_1,\dots,t_r)$ be a vector of linearly independent class functions on $G$ satisfying $\langle t_i,1\rangle=0$ for $1\le i \le r$. Then we have
\[
\delta\bigl(\mathcal{P}_\bt\bigr)=\int_{-\infty}^{-B(t_1)}\dots\int_{-\infty}^{-B(t_r)}f_\bt(x_1,\dots,x_r)\,\mathrm{d} \bx+O_r\left(\left(\frac{\|\widehat{\bt}\|_\infty^{4r}}{V^{2r}}+\frac{\|\widehat{\bt}\|_\infty}{\sqrt{V}}\right)\left(1+\frac{1}{\lambda_\bt}+\frac{1}{\lambda_\bt^r}\right)\right)\, ,
\]    
where $B(t_i)=E(t_i)/\sqrt{V(t_i)}$, $V=\min_{1\leq i \leq r}(V(t_i))$ and $\lambda_\bt$ is the minimal eigenvalue of the matrix $\Delta(\bt)$.
\end{thm}
Our proof of Theorem~\ref{explicit-formula} relies on the recent multidimensional Berry--Esseen inequality of Heuberger--Kropf \cite{HK}. An alternative route would be to adapt Harper--Lamzouri's argument for proving \cite{HL}*{Normal approximation result 1} to the number-field setting, which would replace the error term $1/\sqrt{V}$ with $[L_n:\Q]^{-1/8}$. For many natural infinite families $(L_n)$ the Galois degrees $[L_n:\Q]$ can remain bounded, so an estimate that still decays in this regime is indispensable for results such as Theorem~B. The bound in Theorem~\ref{explicit-formula} meets this requirement. Indeed, in the abelian case Lemma~\ref{lem : pos-real} shows that $V\asymp\log d_L$, which yields an error term $\ll 1/\sqrt{\log d_L}$.\\
In Corollary~\ref{r=1} we recover a result proved by Fiorilli--Jouve for $2$-way races~\cite{FJ}*{Th 5.10}. In Corollary~\ref{3wayraces} we state the $3$-way density formula which, when combined with Lemma~\ref{encadrementrho}, yields the following corollary: 
\begin{cor}\label{abelian-3wayraces}
Let $L/\Q$ be a finite abelian extension with group $G$ for which $\mathsf{GRH}$ and $\mathsf{LI}^-$ hold. Let $a,b,c \in G$ be pairwise distinct. Then,
\[ \delta^{(3)}_{L/\Q}\left( a,b,c \right)=\frac{1}{4}+\frac{1}{2\pi}\arcsin\left(\rho(t_{a,b},t_{b,c})\right)-\frac{1}{2\sqrt{2\pi}}(B(t_{a,b})+B(t_{b,c}))+O\left(\frac{1}{\sqrt{\log d_L}}\right)\, .\]
\end{cor}
In the particular case of the cyclotomic field $L=\Q(\zeta_q)$, we have $\log d_L \sim \varphi(q)\log q$. In this case, Corollary~\ref{abelian-3wayraces} matches the shape of a theorem of Lin--Martin~\cite{LM}*{Theorem 1.5} and applies to all distinct invertible residue classes $a,b,c$ modulo $q$, without the assumption $a^2\equiv b^2 \equiv c^2 \pmod q$ of Lin--Martin. Expanding $\arcsin$ around $-1/2$ recovers Lamzouri's formula~\cite{Lam}*{Corollary 2.3}.\\
We will use Theorem~\ref{explicit-formula} to deduce our other main results. The main difficulty is that one has to bound the terms given by the inverse of the minimal eigenvalue of the covariance matrix $\Delta( \bt)$. In order to explain how one can control these eigenvalue terms, let us fix a family $(L_n)_n$ of Galois extensions of $\Q$ with respective groups $(G_n)_n$ and such that the absolute value of the absolute discriminant $d_{L_n}\longrightarrow \infty$, as $n\to\infty$. We introduce two modes of convergence: 
\begin{defi}\label{defi : uniform convergence}
Let $(\mathcal{V}, \|\cdot\|)$ be a real normed vector space and $r\geq 1$ and, for each $n\geq 1$, fix a map $f_n\, :\, \mathcal{A}_r(G_n)\rightarrow \mathcal{V}\, .$
We say that $(f_n)_n$ converges uniformly to $v\in \mathcal{V}$ if 
\[ \max_{x\in \mathcal{A}_r(G_n)} \| f_n(x) - v \| \underset{n\to \infty}{\longrightarrow} 0\, . \] 
\end{defi}
When the family $(L_n)_n$ of Galois extensions over $\Q$ is increasing (i.e. $L_n\subset L_{n+1}$), we denote by $\pi_n:G_{n+1}\to G_n$ the canonical projection. Note that $\pi_n$ naturally induces a surjective map $\pi_n:G_{n+1}^\#\to G_n^\#$ (where $G^{\sharp}$ denotes the set of conjugacy classes of the group $G$). When $C \in G_{n_0}^\#$ for some sufficiently large $n_0\geq 1$, we denote by $C^{(n)} \subset G_n$ (for $n\geq n_0+1$) a lift of $C$ via the map $\pi_{n-1}\circ\dots\circ\pi_{n_0}$ (which is simply the canonical projection $G_n\to G_{n_0}$). This enables us to introduce the second mode of convergence:
\begin{defi}\label{defi : pointwise convergence}
Let $(L_n)_n$ be an increasing sequence of Galois extensions of $\Q$ with Galois groups $(G_n)_n$. Let $(\mathcal{V}, \|\cdot\|)$ be a real normed vector space and let $r\geq 1$ and for each $n\geq 1$ fix a map $f_n\, :\, \mathcal{A}_r(G_n)\rightarrow \mathcal{V}\, .$
We say that $(f_n)_n$ converges pointwise to $v\in \mathcal{V}$ if for all $n_0\geq 1$ and for all $(C_1,\dots,C_r)\in \mathcal{A}_r(G_{n_0})$ and for any choice of lifts $(C_1^{(n)},\dots,C_r^{(n)})\in \mathcal{A}_r(G_n)$, $n\geq 1$, we have 
\[ \lim_{n\to \infty} f_n\left(  C_1^{(n)},\dots,C_r^{(n)} \right) =v \, .\]
\end{defi}
We note that a family $(L_n)_n$ is uniformly $r$-moderate if and only if $\bigl(\delta^{(r)}_{L_n/\Q})_n$ converges uniformly to $1/r!$. Similarly, we define the notion of pointwise moderacy.

\begin{defi}\label{def:pointwise}
    We say that $(L_n)_n$ is pointwise $r$-moderate if $\bigl(\delta^{(r)}_{L_n/\Q})_n$ converges pointwise to $1/r!$.
\end{defi}

 In order to bound the eigenvalue terms in the explicit formula Theorem \ref{explicit-formula}, it suffices to prove that $\Delta_{L_n/\Q}$ converges uniformly to a symmetric matrix that is positive-definite. Sometimes, proving a pointwise convergence is enough to bound those eigenvalues for some specific families of conjugacy classes. In fact, we prove in Proposition~\ref{matrix-conv}, that the pointwise convergence of this sequence of matrices always holds, leading to the following result: 
\begin{thm}\label{thm : pointwisemoderacy}
    Let $(L_n)_n$ be an \emph{increasing} family of abelian extensions of $\Q$ satisfying $\mathsf{GRH}$ and $\mathsf{LI}$, with respective groups $(G_n)_n$ and such that $[L_n : \Q]\to \infty$. For all $r\geq 2$ the following are equivalent:
    \begin{enumerate}\itemsep0.5em
        \item $r(G_n)/\sqrt{\log d_{L_n}} \underset{n \to \infty}{\longrightarrow} 0$.
        \item $(L_n)_n$ is uniformly $2$-moderate over $\Q$.
        \item $(L_n)_n$ is pointwise $r$-moderate over $\Q$.
    \end{enumerate}
\end{thm}
The behavior in terms of uniform convergence of the latter functions is related to the ramification data of the corresponding extensions. When $L/\Q$ is a Galois extension with group $G$ and $\chi$ is a character of $G$, denote by $A(\chi)$ its Artin conductor (see \S~\ref{subsec:artin} for a definition). The following Theorem gives a full characterization, in the abelian case, of uniformly $r$-moderate extensions when $r\ge 3$. 
\begin{thm}\label{thm : uniform-moderacy}
    Let $(L_n)_n$ be a family of abelian Galois extensions over $\Q$ with respective Galois groups $(G_n)_n$, such that  $d_{L_n/\Q}\to \infty$ as $n\to \infty$, and for which $\mathsf{GRH}$ and $\mathsf{LI}$ hold. Let $r\ge 3$, we have: $(L_n)_n$ is uniformly $r$-moderate if and only if $(L_n)_n$ is uniformly $2$-moderate and \[ \lim_{n\to \infty} \frac{1}{\log d_{L_n}}\max_{a,b\in G_n\setminus\{1\}} \Bigl| \sum_{\chi \in \Irr(G_n)}(\chi(a)-\chi(b)) \log A(\chi)\Bigr|=0\, .\]
\end{thm}
As an instance of applying Theorem~\ref{thm : uniform-moderacy}—more examples of which will be given in \S\ref{subsec:uniformmod}—we have the following corollary:

\begin{cor}\label{primeorder-extensions}
    Let $(L_n)_n$ be a family of abelian Galois extensions of $\Q$ for which $\mathsf{GRH}$ and $\mathsf{LI}$ hold, with respective Galois groups $(G_n)_n$ all cyclic with prime order, and such that $d_{L_n/\Q}\to \infty$ as $n\to \infty$. Then, $(L_n)$ is uniformly $r$-moderate for all $r\ge 2$.
\end{cor}
A central auxiliary quantity in our study is the following: 

\begin{align}U_{L/\Q}\colon\;G\setminus \{1\}&\longrightarrow \R\\ \nonumber
\quad a &\longmapsto \frac{1}{N_{L}} \sum_{\chi \ne 1}\sum_{\gamma_\chi>0}\frac{2 \Re(\chi(a))}{\frac{1}{4}+\gamma_\chi^2}
\end{align}

where \begin{equation}\label{def : NL}
    N_{L}=2\sum_{\chi \ne 1}\sum_{\gamma_\chi>0}\frac{1}{\frac{1}{4}+\gamma_\chi^2}.
\end{equation}
Our main density result, which gives the exact criterion leading to Theorems C and D, can be stated in terms of the function $U$ as follows: 
\begin{thm}\label{main-density}
Let $(L_n)_n$ be an \emph{increasing} family of abelian extensions of $\Q$, for which $\mathsf{GRH}$ and $\mathsf{LI}$ hold, with respective Galois groups $(G_n)_n$, and such that $[L_n : \Q]\to \infty$. Assume that one of the following holds:
\begin{itemize}
    \item The set of limit points of $(r(G_n)/\sqrt{\log d_{L_n}})_n$ has non-empty interior.
    \item We have $r(G_n)/\sqrt{\log d_{L_n}} \longrightarrow 0$ and the set 
$$\overline{ \bigcup_{n\geq1} \left|U_{L_n/\Q}\right|(G_n\setminus\{1\})}$$
has non-empty interior.
\end{itemize}
Then, for all $r\geq 3$, the set
$$\overline{\bigcup_{n\geq 1} \delta_{L_n/\Q}^{(r)}(\mathcal{A}_r(G_n))} $$
has non-empty interior. 
If moreover, for some $\ve>0$ we have $(1-\ve\, ,1)\subset \overline{ \bigcup_{n\geq1} \left|U_{L_n/\Q}\right|(G_n\setminus\{1\})}$, and $r(G_n)/\sqrt{\log d_{L_n}} \longrightarrow 0$, and if $r\ge 4$, then $$0\in \overline{\bigcup_{n\geq 1} \delta_{L_n/\Q}^{(r)}(\mathcal{A}_r(G_n))}\, .$$
\end{thm}

\section{Number theoretic tools}\label{sec:tools}
\subsection{Artin conductors}\label{subsec:artin}
In this subsection, we prove some basic properties of Artin conductors in the abelian case, which will be crucial for our study. Let $L/K$ be a Galois extension of number fields (not necessarily abelian) with Galois group $G$. Let $\frak{p}$ be a prime of $K$ and let $\frak{P}$ be a prime of $L$ lying above $\frak{p}$. Let $\chi$ be the character of a representation $\rho: G \to \GL(V)$. We then define 
\[n(\chi,\frak{p}):= \sum_{i\geq 0} \frac{|G_i(\frak{P}/\frak{p})|}{|G_0(\frak{P}/\frak{p})|}\text{codim} V^{G_i(\frak{P}/\frak{p})}\, ,\]
where $(G_i(\frak{P}/\frak{p}))_{i\geq0}$ is the sequence of higher ramification groups and $V^{G_i(\frak{P}/\frak{p})}$ is the $G_i(\frak{P}/\frak{p})$-invariant subspace of $V$. Artin proved that $n(\chi,\frak{p})$ is an integer. The Artin conductor of $\chi$ is by definition the ideal of $\mathcal O_K$ defined by \[\frak{f}(L/K,\chi)=\prod_{\frak{p}}\frak{p}^{n(\chi,\frak{p})}.\]
Define \[A(\chi)=d_K^{\chi(1)}N_{K/\Q}\left(\frak{f}(L/K,\chi) \right).\]
The conductor-discriminant formula states that the relative discriminant $D_{L/K}$ satisfies
\begin{equation}\label{conductor-discriminant}
    D_{L/K}=\prod_{\chi \in \Irr(G)} \frak{f}(L/K,\chi)^{\chi(1)}.
\end{equation}
As in \cite{FJ}*{Section 4}, we will use the following convenient formula for $n(\chi,\frak{p})$:
\begin{equation}\label{form-cond} n(\chi,\frak{p})=\frac{1}{|G_0(\frak{P}/\frak{p})|}\sum_{i\geq 0} \sum_{b\in G_i(\frak{P}/\frak{p}) } (\chi(1)-\chi(b^{-1})).\end{equation}
We first state an important result due to Jonah Leshin that will be crucial in what follows. We restate it in a slightly different way: 
\begin{thm}[\cite{Leshin}*{Theorem 1}]\label{rtdiscbounded}
For any number field $K$ and constant $C > 0$, there are only finitely many finite abelian extensions $L/K$ such that 
$$\frac{\log d_L}{[L:\Q]} \le C.$$
\end{thm}
In fact, \cite{Leshin}*{Theorem 1} is more general in the sense that we might replace "abelian" by "solvable with bounded length", but the main part of its proof is the abelian case, which will be enough for our purposes.\\
A first consequence of this result is the following:
\begin{cor}\label{NLdL}
    Let $L/\Q$ be a finite abelian extension with group $G$ for which $\mathsf{GRH}$ holds. Then 
    $$\sum_{\chi \ne 1} \log \log A(\chi) =o (\log d_L) \quad (d_L\to \infty). $$
    As a consequence, \[ N_L\sim \log d_L\quad (d_L\to \infty).\]
\end{cor}
\begin{proof}
Using the arithmetic mean-geometric mean (AM-GM) inequality combined with \eqref{conductor-discriminant}, we deduce
\begin{align*}
    \frac{\sum_{\chi\ne 1}\log \log A(\chi)}{|G|-1}&=\log \left(\left( \prod_{\chi \ne 1}\log (A(\chi))\right)^{1/(|G|-1)}\right) \\
    &\leq \log \left (\frac{\log d_L}{|G|-1}\right).
\end{align*}
Thus, 
$$ \frac{\sum_{\chi\ne 1}\log \log A(\chi)}{\log d_L} \leq \frac{\log ((\log d_L)/(|G|-1))}{(\log d_L)/(|G|-1)}. $$
Since by Theorem~\ref{rtdiscbounded}  $$ \frac{\log d_L}{|G|} \longrightarrow \infty \quad (d_L\to \infty, L\text{ abelian})\, ,$$
we have $$\sum_{\chi \ne 1} \log \log A(\chi) = o(\log d_L) \quad (d_L\to \infty)\, .$$
This proves the first part. For the second part, we use~\cite{LO}*{ (5.11)} and~\cite{Ng}*{Proposition 2.4.2.3} to deduce that for all $\chi \ne 1$
\begin{equation}\label{LO}
    \sum_{\gamma_\chi\ne0}\frac{1}{\frac{1}{4}+\gamma_\chi^2}=\log A(\chi)+O(\log \log A(\chi) )\, .
\end{equation}
This implies, using \eqref{conductor-discriminant} again, that $$N_L=\sum_{\chi\ne 1}\sum_{\gamma_\chi\ne0}\frac{1}{\frac{1}{4}+\gamma_\chi^2}=\log d_L +o(\log d_L)\quad (d_L\to \infty)\, .$$
\end{proof}
For an abelian Galois extension $L/\Q$ with group $G$ and $a\in G\setminus \{1\}$, our goal will be to estimate the following quantity 
$$\sum_{\chi \in \Irr(G)} \chi(a)\log A(\chi) \, ,$$
which is well-studied in the classical case of prime numbers, or equivalently, when $L$ is a cyclotomic extension (see for instance, \cite{FiM}*{Proposition 3.3}). A first observation is that it is easy to obtain the contribution of all $a\in G\setminus \{1\}$:
\begin{lem}\label{contributionalla's}
Let $L/\Q$ be an \emph{abelian} Galois extension of number fields with group $G$. We have 
$$\sum_{a\in G\setminus\{1\}} \sum_{\chi \ne 1} \chi(a) \log A(\chi)=-\log d_L\, .$$
\end{lem}
\begin{proof}
This is an immediate consequence of the orthogonality relation $\sum_{a \ne 1} \chi(a)=-1$ when $\chi \ne 1$, combined with \eqref{conductor-discriminant}.
\end{proof}
Another application of the orthogonality relations combined with \eqref{form-cond} yields the following Lemma.
\begin{lem}\label{negativity}
    Let $L/K$ be an abelian extension of number fields with group $G$, and let $\frak{p}$ be a prime of $K$ and $a\in G\setminus\{1\}$. We have 
    \[
    \sum_{\chi \in \Irr(G)}\chi(a) n(\chi,\mathfrak{p})=-\frac{|G|}{|G_0(\frak{P}/\frak{p})|}\#\left\{i\geq0\ :\ a\in G_i(\frak{P}/\frak{p}) \right\}\, .
    \]
\end{lem}
\begin{proof}
    Using~\eqref{form-cond} we deduce that \[ \sum_{\chi \in \Irr(G)}\chi(a) n(\chi,\mathfrak{p})=\frac{1}{|G_0(\frak{P}/\frak{p})|}\sum_{i\geq 0} \sum_{b\in G_i(\frak{P}/\frak{p}) } \sum_{\chi \in \Irr(G)}(\chi(a)-\chi(ab^{-1})) \, . \]
    Since $a\ne 1$, the lemma follows by orthogonality relations.
\end{proof}
The previous Lemma is important; as a direct consequence we have: 
$$ \sum_{\chi \in \Irr(G)}\chi(a) \log A(\chi) \in \R_{\le 0} \, . $$
This negativity will be crucial in studying the variance because of the following:
\begin{cor}\label{negativity of U(a)}
Let $L/\Q$ be an abelian extension with group $G$ and for which $\mathsf{GRH}$ holds. Then for all $a\in G\setminus\{1\}$ \[ U_{L/\Q}(a)=\frac{1}{N_L}\sum_{\chi \in \Irr(G)} \chi(a) \log A(\chi)+ o(1) \quad (d_L\to \infty) \, .\]
\end{cor}
\begin{proof}
Immediate by Corollary~\ref{NLdL} and equation \eqref{LO}.
\end{proof}

We now state and prove the main technical result of this section:
\begin{lem}\label{UbllG}
    Let $K\subset L$ be two abelian Galois extensions over $\Q$ with respective groups $G$ and $G^+$ and let $b \in G^+$ be an element whose restriction $a\in G$ is non-trivial. Then, 
    \[
    \left|\sum_{\psi \in \Irr(G^+)}\psi(b) \log A(\psi) \right| \le 2 |G^+| \log d_{K}\, .
    \]
\end{lem}
\begin{proof}
    We first note that $$\sum_{\psi \in \Irr(G^+)}\psi(b) \log A(\psi)=\sum_p\left (\sum_{\psi \in \Irr(G^+)}\psi(b) n(\psi,p) \right)\log p \, ,$$
    where the first sum runs over prime numbers $p$ that ramify in $L$. Our first goal will be to prove that for all $p$ ramifying in $L$ we have $$ \left|\sum_\psi \psi(b) n(\psi,p)\right| \le \begin{cases}
       |G^+|-\sum_{\chi \in \Irr(G)}\chi(a) n(\chi,p) \quad \text{if } p\text{ ramifies in } K \\
       0\qquad \ \text{otherwise}
    \end{cases} \, . $$
    Let $p$ be a prime number ramifying in $L$, let $\frak{p}$ be a prime of $K$ lying above $p$, and let $\frak{P}$ be a prime of $L$ lying above $\frak{p}$ with ramification index $e$. Note that \begin{equation*}
        e\left|G_0(\frak{p}/p)\right|=\left|G_0(\frak{P}/p)\right|\, .
    \end{equation*}
    If $a\notin G_0(\frak{p}/p)$, then $b\notin G_0(\frak{P}/p)$: indeed, if $b\in G_0(\frak{P}/p)$ then for all $x\in O_L$, $bx-x\in \frak{P}$ then for all $x\in O_K$ we have $ax-x\in \frak{P}\cap O_K=\frak{p}\, ,$ thus $a\in G_0(\frak{p}/p)$ which contradicts the assumption that $a$ is not in $G_0(\frak{p}/p)$. This proves in particular that if $p$ does not ramify in $K$ then by Lemma~\ref{negativity}, we conclude that $\sum_\psi \psi(b) n(\psi,p)=0 \, .$\\
    Assume now that $a\in G_0(\frak{p}/p)$. If $b\notin G_{e-1}(\frak{P}/p)$ then applying Lemma~\ref{negativity}
    \begin{align*}
        -\sum_\psi \psi(b) n(\psi,p)&=\frac{|G^+|}{|G_0(\frak{P}/p)|} \# \{i\geq 0\ :\ b\in G_i(\frak{P}/p)\}\leq \frac{|G^+|}{|G_0(\frak{P}/p)|}e\\
        &=\frac{|G^+|}{|G_0(\frak{p}/p)|}\leq \frac{|G^+|}{|G_0(\frak{p}/p)|}\#\{i\geq 0\ :\ a\in G_i(\frak{p}/p)\}\\
        &=-\frac{|G^+|}{|G|} \sum_\chi \chi(a)n(\chi,p) \, .
    \end{align*}
    Assume that $b\in G_{e-1}(\frak{P}/p)$ and let $s:=\max\{ i\geq 1\ :\ b\in G_{ie-1}(\frak{P}/p) \}\, .$
    As $b\notin G_{(s+1)e-1}$ by Lemma~\ref{negativity}
    \begin{align*}
        -\sum_\psi \psi(b) n(\psi,p)&=\frac{|G^+|}{|G_0(\frak{P}/p)|} \# \{i\geq 0\ :\ b\in G_i(\frak{P}/p)\}\\
        &< \frac{|G^+|}{|G_0(\frak{P}/p)|}(s+1)e=\frac{|G^+|}{|G_0(\frak{p}/p)|}(s+1)\, .
    \end{align*}
    Since $b\in G_{se-1}(\frak{P}/p)$, thus, for all $x\in O_K$
    $$ax-x=bx-x\in O_K\cap \frak{P}^{se}=\frak{p}^s \, ,$$
    and since the sequence of ramification groups is decreasing, we have
    $$ \#\{i\geq 0\ :\ a\in G_i(\frak{p}/p)\} \geq s \, .$$
    Hence, 
    \begin{align*}
        -\sum_\psi \psi(b) n(\psi,p)&\leq \frac{|G^+|}{|G_0(\frak{p}/p)|}(\#\{i\geq 0\ :\ a\in G_i(\frak{p}/p)\}+1)\\
        &\le |G^+|-|G^+| \sum_{\chi \in \Irr(G)}\chi(a) n(\chi,p) \, .
    \end{align*}
    Thus, by \eqref{conductor-discriminant} and summing over ramified primes $p$, we obtain
    \[ \Bigl|\sum_{\psi \in \Irr(G^+)}\psi(b) \log A(\psi) \Bigr| \leq 2|G^+|\log d_K,\]
    which proves the result.
    
\end{proof}

\subsection{Multiquadratic extensions}\label{subsec:multiquadratic}
Because of their importance to our work, we will recall some basic facts on multiquadratic extensions and provide some related density results. Let $(p_n)_n$ be an increasing sequence of prime numbers that are congruent to $1$ mod $4$. Define \[ L_n=\Q\left(\sqrt{p_1},\dots,\sqrt{p_n}\right)\quad (n\ge 1) \, .\]
We have $G_n:=\Gal(L_n/\Q)\simeq (\Z/2\Z)^n $. Moreover, if $\sigma_i \in G_n$ is the automorphism associated to $p_i$, that is $\sigma_i\bigl(\sqrt{p_j}\bigr)=(-1)^{\delta_{i,j}}\sqrt{p_j}$, with $\delta_{i,j}=1$ if $i=j$ and $\delta_{i,j}=0$ otherwise, then the inertia group of $L_n$ at $p_i$ is given by
\[ (G_n)_0\bigl(p_i\bigr)=\langle \sigma_i \rangle \simeq \Z/2\Z \, . \]
The ramification at each prime $p_i$ is tame, since $\gcd(p_i\, , |G_n|)=1$, and there are no other ramified primes since $L_n \subset \Q(\zeta_{q_n})$ with $q_n=p_1\dots p_n$ (that is because $p_i$ is a square in $\Q(\zeta_{p_i})$).  \\
By Lemma~\ref{negativity}, we have for all $1\le i \le n$
\begin{align*}
    \sum_{\chi \ne 1}\chi(\sigma_i) \log A(\chi)&=-\sum_{j=1}^n \left(\frac{|G_n|}{|(G_n)_0(p_j)|}\#\{\,  \ell\ge 0\,  \colon \sigma_i\in (G_n)_\ell(p_j)\, \}\right) \log p_j\\
    &=-\frac{|G_n|}{2}\log p_i\, .
\end{align*} 
By Lemma~\ref{contributionalla's}, we deduce that 
\[ \log d_{L_n}=\frac{|G_n|}{2} \sum_{i=1}^n \log p_i\, .\]
In order to state our first density result, we will need the following Lemma: 
\begin{lem}\label{lem : prime-density}
    For all $\ell\ge 5$ and all $\alpha \in (0,1)$, there exist primes $p_m>\dots>p_1>\ell$, all congruent to $1 \,\mathrm{mod}\, 4$, such that 
    \[ \left| \frac{\log p_m}{\log \ell+\log p_1+\dots+\log p_m}-\alpha \right|\,  \le \frac{\log 2}{\log\bigl(\ell\, p_1\dots p_{m-1}\bigr)}\, . \]
\end{lem}
\begin{proof}
   Let $\ell\ge 5$ and $0<\alpha<1$. Denote $\theta=\frac{\alpha}{1-\alpha}$ so that $\alpha=\frac{\theta}{1+\theta}$. Using the prime number theorem in arithmetic progressions, we deduce that for any sufficiently large $m$, taking consecutive primes $p_{m-1}>\dots>p_1>\ell$, all congruent to $1$ mod $4$, we have
   \begin{equation}\label{consec-primes} \theta \cdot \bigl(\log \ell +\log p_1+\dots+\log p_{m-1}\bigr)>\log p_{m-1}\, .\end{equation}
   By Bertrand's postulate in arithmetic progressions, for all $x$ large enough, there exists a prime $p\in (x\, , 2x)$ such that $p\equiv 1\pmod 4$. In particular, up to taking a larger $m$, there exists a prime $p_m\equiv 1\pmod 4$ such that 
   \[ \bigl(\ell\, p_1\dots p_{m-1}\bigr)^\theta \, < p_m <2\,\bigl(\ell\, p_1\dots p_{m-1}\bigr)^\theta\, . \]
   By \eqref{consec-primes} we have $p_m>p_{m-1}$. It suffices to verify the desired inequality: 
   \begin{align*}
       \left| \frac{\log p_m}{\log \ell+\log p_1+\dots+\log p_m}-\frac{\theta}{1+\theta} \right|&=\left|\frac{\log p_m\, -\theta\cdot(\log \ell+\dots+\log p_{m-1})}{(1+\theta)(\log \ell+\cdots+\log p_m)}\right|\\
       &\le \frac{\log 2}{(1+\theta)(\log \ell\, p_1\dots p_{m-1})}\, ,
   \end{align*}
   the lemma follows, since $\theta>0$.   
\end{proof}

\begin{prop}\label{multiquadratic-U-density}
There exists an increasing sequence of primes $(p_n)_n$ all congruent to $1$ mod $4$ such that, if $L_n=\Q\bigl(\sqrt{p_1},\dots,\sqrt{p_n}\bigr)$, and if $\mathsf{GRH}$ holds for each $L_n$, then 
\[ \overline{ \bigcup _{n\ge 1} \bigl| U_{L_n/\Q}\bigr| \bigl(G_n\setminus\{1\}\bigr)}=[0\, ,1].\]
\end{prop}
\begin{proof}
We first note that for any increasing sequence of primes $(p_n)_n$ that are congruent to $1$ mod $4$, if $\sigma_n$ is the automorphism associated to $p_n$ (as above), then, by Lemma~\ref{NLdL} and Corollary~\ref{negativity of U(a)}, we have
\[ |U_{L_n/\Q}|(\sigma_n)=\frac{\log p_n}{\log p_1+\dots+\log p_n}+o(1)\quad (n\to \infty) \, .\]
Since a $o(1)$ does not change the set of limit points, it suffices to construct a sequence $(p_n)_n$ such that the sequence \[ \left(\frac{\log p_n}{\log p_1+\dots+\log p_n}\right)_n \] is dense in $[0,1]$.
The idea is as follows, fix a dense sequence $(d_n)_n$ in $(0,1)$ (take, for instance, a sequence listing all elements of $\Q \cap (0,1)$). We want to apply Lemma~\ref{lem : prime-density} at each $d_n$; first denote $\ell_1:=5\equiv1\pmod 4$ and apply Lemma~\ref{lem : prime-density} $\ell = \ell_1$ and $\alpha = d_1$, which ensures the existence of $m_1>2$, and primes $p_{1,m_1}>\dots>p_{1,1}>\ell_1$ satisfying: \[\left|\frac{\log p_{1,m_1}}{\log \ell_1+\log p_{1,1}+\dots+\log p_{1,m_1}}-d_1 \right|\, \le \frac{\log 2}{\log \bigl(\ell_1\, p_{1,1}\dots p_{1,m_1-1}\bigr)}\, .\]
Define $\ell_2:=\ell_1\, p_{1,1}\dots p_{1,m_1-1}$. Assume that $n\ge 2$ and that $m_1,\dots,m_{n-1}>2$, $\ell_1,\dots,\ell_{n-1}$ and primes $p_{n-1,m_{n-1}}>\dots>p_{n-1,1}>\dots>p_{1,m_1}>\dots>p_{1,1}>\ell_1$ are constructed. Define $\ell_n$ as the product of all previous primes (including $\ell_1$) and apply Lemma~\ref{lem : prime-density} on $\ell_n$ and $d_n$, to deduce the existence of primes $p_{n,m_n}>\dots>p_{n,1}>\ell_n$ such that 
\[\left|\frac{\log p_{n,m_n}}{\log \ell_n+\log p_{n,1}+\dots+\log p_{n,m_n}}-d_n \right|\, \le \frac{\log 2}{\log \bigl(\ell_n\, p_{n,1}\dots p_{n,m_n-1}\bigr)}\, ,\]
this gives the construction of the desired sequence of primes.
\end{proof}
We now move to our second density result which is in the same spirit as a result proved by Fiorilli~\cite{Fiorilli}*{Lemma 3.3}. We will be brief in the proof since it is similar to the proof of Proposition~\ref{multiquadratic-U-density}.
\begin{prop}\label{multiquadratic-B-density}
    There exists an increasing sequence of primes $(p_n)_n$ all congruent to $1 \,\mathrm{mod}\, 4$ such that, if $L_n=\Q\bigl(\sqrt{p_1},\dots,\sqrt{p_n}\bigr)$ and if $G_n=\Gal(L_n/\Q)$, then the sequence 
    \[\left(\frac{r(G_n)}{\sqrt{\log d_{L_n}}}\right)_n\]
    is dense in $\R_{>0}$.
\end{prop}
\begin{proof}
    We first note that \[\frac{r(G_n)}{\sqrt{\log d_{L_n}}}=\sqrt{2}\left(\frac{2^n}{\log p_1+\dots+\log p_n}\right)^{1/2}.\]
    Thus, it suffices to construct $(p_n)_n$ so that the sequence $\bigl(2^n/\log (p_1\dots p_n) \bigr)_n$ is dense in $\R_{>0}$. As in the proof of Proposition~\ref{multiquadratic-U-density} it suffices to prove that if $x>0$ and $q_1<\dots<q_m$ are given prime numbers, then for all $\ve>0$, there exist prime numbers $p_n>\dots>p_1>q_m$ such that 
    \begin{equation}\label{eq-densityobj}
        \left| \frac{2^{n+m}}{\log (q_1\dots q_m p_1\dots p_n)}-x\right|<\ve.
    \end{equation} 
    Indeed, once this holds, one might consider a sequence of positive real numbers $(r_n)_n$ that is dense in $\R_{>0}$, and apply~\eqref{eq-densityobj} inductively with $x=r_n$ and $\ve=1/n$ for each $n\ge 1$. Fix $x\in \R_{>0}$ and $q_m>\dots>q_1$ prime numbers and fix $\ve>0$ small enough so that $x-\ve>0$.
    Note that~\eqref{eq-densityobj} is equivalent to 
    \begin{equation}\label{eq-density}
        \frac{2^{n+m}}{x+\ve}-(\log q_1\dots q_mp_1\dots p_{n-1})< \log p_n < \frac{2^{n+m}}{x-\ve}-(\log q_1\dots q_mp_1\dots p_{n-1})\, .
    \end{equation}
    By the prime number theorem in arithmetic progressions, we can consider consecutive primes $p_{n-1}>\dots>p_1>q_m$, that are congruent to $1$ mod $4$, with $n$ large enough so that
    \[\frac{2^{n+m}}{x+\ve}-(\log q_1\dots q_mp_1\dots p_{n-1})> \log p_{n-1}\, ,\]
    and \[ \exp\left(2^{n+m} \left(\frac{1}{x-\ve}-\frac{1}{x+\ve}\right)\right)>2\, .\]
    This implies, using again Bertrand's postulate in arithmetic progressions, that there exists a prime $p_n\equiv 1\pmod 4$ satisfying~\eqref{eq-density}, which proves~\eqref{eq-densityobj}.
\end{proof}
A question that arises is: are the previous two density questions related? We give the following partial answer which is enough for our purposes:
\begin{prop}\label{2mod-Udense}
    There exists an increasing sequence of primes $(p_n)_n$ all congruent to $1 \,\mathrm{mod}\, 4$ such that, denoting $L_n=\Q\bigl(\sqrt{p_1},\dots,\sqrt{p_n}\bigr)$ and $G_n=\Gal(L_n/\Q)$, we have $$r(G_n)/(\log d_{L_n})^{1/2} \longrightarrow 0,$$ as $n\to \infty$, and \[ \overline{ \bigcup _{n\ge 1} \bigl| U_{L_n/\Q}\bigr| \bigl(G_n\setminus\{1\}\bigr)}=[1/2 ,1].\]
\end{prop}
\begin{proof}
    If $\alpha>1/2$ and $\theta=\frac{\alpha}{1-\alpha}$, then $\theta>1$. Thus, given any prime numbers $q_1<\dots<q_m$ we can choose $p_1\equiv 1\pmod 4$ a prime large enough so that: 
    \[\frac{2^{m+1}}{\log q_1\dots q_m\, p_1}<\frac{1}{2m}\, .\]
    We note that as $\theta>1$, then \begin{equation} \theta \cdot \bigl(\log q_1\dots q_m +\log p_1 \bigr)>\log p_{1}\, ,\end{equation}
    so that we can choose $p_2\equiv 1 \pmod 4$ satisfying 
    \[ \bigl( q_1\dots q_m p_1\bigr)^\theta \, < p_2 <2\,\bigl( q_1\dots q_m p_1\bigr)^\theta\, , \]
    which proves that, as in the proof of Lemma~\ref{lem : prime-density}, 
    \[ \left| \frac{\log p_2}{\log q_1+\dots+\log q_{m}+\log p_1+\log p_2}-\alpha \right|\,  \le \frac{\log 2}{\log\bigl( q_1\dots q_{m}p_1\bigr)}\, . \] 
    Applying this inductively, as in the proof of Proposition~\ref{multiquadratic-U-density}, gives the result.
\end{proof}

\section{Explicit formulas for logarithmic densities}\label{sec:explicitformula}

\subsection{A multidimensional Berry--Esseen inequality}
Let \(r \geq 1\), and let us denote by \(L = \{1,\dots,r\}\) the set indexing the coordinates in \(\R^r\). Given a vector \(\bs \in \R^r\) and a subset \(K \subseteq L\), we denote by \(\bs_K = (s_j)_{j \in K} \in \R^K\) the restriction (or projection) of \(\bs\) to the coordinates in \(K\). For \(J \subseteq K\), we define the projection operator \(\psi_{J,K}\), which zeros out the coordinates outside \(J\), i.e., 
\[
\psi_{J,K}((s_j)_{j \in K}) = (u_k)_{k \in K} \quad \text{with } u_k = \begin{cases}
s_k & \text{if } k \in J, \\
0 & \text{otherwise}.
\end{cases}
\]
Let \(\Pi_K\) denote the set of all \emph{partitions} of the finite set \(K\). If \(\alpha = \{J_1,\dots,J_\ell\} \in \Pi_K\) is a partition, we denote its size by \(|\alpha| = \ell\). We now define a non-linear operator that will play a central role in the formulation of the Berry--Esseen inequality.
\begin{defi}\label{deltaop}
Let \(K \subseteq \mathbb N\) be finite and \(h : \R^K \to \R\). We define
\[
\Lambda_K(h) := \sum_{\alpha \in \Pi_K} \mu_\alpha \prod_{J \in \alpha} h \circ \psi_{J,K},
\]
where the coefficient \(\mu_\alpha\) is given by
\[
\mu_\alpha := (-1)^{|\alpha|-1} (|\alpha|-1)!.
\]
\end{defi}
We shall write \(\Lambda := \Lambda_K\) when \(K=L=\{1,\dots,r\}\). For example, in dimension $2$ (taking $L=\{1,2\}$), we have \[ \Lambda (h)(s_1,s_2)=h(s_1,s_2)-h(s_1,0)h(0,s_2)\quad\left((s_1,s_2) \in \R^2 \right)\, .\]
We can now state the following multidimensional Berry--Esseen inequality due to Heuberger and Kropf~\cite{HK}. 

\begin{thm}[Berry--Esseen inequality in dimension \(r\), \cite{HK}]\label{Berry-Esseen}
Let \(Y\) and \(Z\) be two \(r\)-dimensional random variables and let \(F_Y\) and \(F_Z\) be their respective cumulative distribution functions. Suppose that \(F_Z\) is differentiable such that $M:=\max_{i\le r}\|\partial F_Z/\partial x_i \|_\infty<\infty $. Then for every \(T > 0\), one has:
\begin{equation*} \label{BerryEsseen}
\sup_{\bz \in \R^r} \left| F_Y(\bz) - F_Z(\bz) \right|
\ll_r \int_{\|\bs\| \leq T} 
\left| \frac{\Lambda(\varphi_Y)(\bs) - \Lambda(\varphi_Z)(\bs)}{\prod_{j=1}^r s_j} \right| \, \mathrm{d}\bs+ \sum_{\emptyset \neq J \subsetneq L}  \sup_{\bz_J \in \R^J} \left| F_{Y_J}(\bz_J) - F_{Z_J}(\bz_J) \right|+\frac{M}{T} 
\end{equation*}
where \(Y_J\) and \(Z_J\) denote the marginals of \(Y\) and \(Z\) on the coordinates in \(J\), and \(\varphi_Y\), \(\varphi_Z\) are the characteristic functions of \(Y\) and \(Z\), respectively.
\end{thm}

\subsection{Preliminary results}

In this subsection, we fix a Galois extension $L/\Q$ with group $G$ and for which $\mathsf{GRH}$, $\mathsf{AC}$, and $\mathsf{LI}$$^{-}$ hold. Let $\bt=(t_1,\dots,t_r)$ be a vector of linearly independent class functions on $G$ satisfying $\langle t_i,1\rangle=0$ for $1\le i \le r$. Denote $\|\widehat{\bt}\|_\infty=\max_{1\le i \le r}\max_\chi |\langle t_i,\chi\rangle | $ and define $$\bB(\bt)=\left(\frac{\bE(t_1)}{\sqrt{V(t_1)}},\dots,\frac{\bE(t_r)}{\sqrt{V(t_r)}}\right)\, .$$
\begin{prop}\label{TCL}
    Denote $V:=\min_{1\leq i \leq r}V(t_i)$, and let $\bx=(x_1,\dots,x_r)\in \R^r$ such that $$\|\bx\|\leq \frac{\sqrt{V}}{4r\|\widehat{\bt}\|_\infty}\, .$$ 
    We have \[
    \widehat{\mu}\left(\frac{x_1}{\sqrt{V(t_1)}},\dots,\frac{x_r}{\sqrt{V(t_r)}}\right)=e^{-i\langle \bB(\bt),\bx\rangle}\exp\left(-\frac{x_1^2+\dots+x_r^2}{2}-\sum_{1\leq i<j \leq r} \rho(t_i,t_j)x_ix_j \right)F_\bt (x_1,\dots,x_r) \, ,
    \]
    where $F_\bt(x_1,\dots,x_r)=\exp\left(-\sum_{n\geq2}b_{2n}(\bx;\bt)\right)$, and there exists $0<a_{2n}\ll(5/12)^{2n}$ such that for $n\geq 1$
    \[
    b_{2n}(\bx;\bt):=2^{2n}a_{2n}\sum_{\chi\ne1}\sum_{\gamma_\chi>0}\frac{\left|\sum_{i=1}^r \langle t_i,\chi\rangle x_i V(t_i)^{-1/2}\right|^{2n}}{\left(\frac{1}{4}+\gamma_\chi^2\right)^n}
    \] 
\end{prop}
\begin{proof}
Denote \[ G(x_1,\dots,x_r):=\sum_{\chi \ne 1} \sum_{\gamma_\chi>0} \log \J \left(\frac{2\bigl|\sum_{i=1}^r \widehat{t_i}(\chi)\frac{x_i}{\sqrt{V(t_i)}} \bigr|}{\sqrt{\frac{1}{4}+\gamma_\chi^2}}\right).\]
By~\cite{FiM}*{Lemma 2.8}, we have
\[ \log \J (s)=-\sum_{n\geq1}a_{2n}s^{2n}\quad (|s| \le 1 )\, , \]
with $a_2=1/4$ and $0<a_{2n}\ll (5/12)^{2n}$. We deduce that 
\[G(x_1,\dots,x_r)=-\sum_{n\geq1} b_{2n}(\bx;\bt) \, .\]
Computing the case $n=1$ we obtain 
\begin{align*}
    b_2(\bx;\bt)&=\sum_{\chi\ne1}\sum_{\gamma_\chi>0}\frac{\left|\sum_{i=1}^r \langle t_i,\chi\rangle x_i V(t_i)^{-1/2}\right|^{2}}{\frac{1}{4}+\gamma_\chi^2}\\
    &=\sum_{i=1}^r \frac{x_i^2}{V(t_i)}\sum_{\chi\ne1}|\widehat{t_i}(\chi)|^2\sum_{\gamma_\chi>0}\frac{1}{\frac{1}{4}+\gamma_\chi^2}+\sum_{1\leq i<j \leq r} \rho(t_i,t_j)x_ix_j\\
    &=\frac{x_1^2+\dots+x_r^2}{2}+\sum_{1\leq i<j \leq r} \rho(t_i,t_j)x_ix_j\, .
\end{align*}
This proves the result, since by~\eqref{fourier-transform} we have
\[ \widehat{\mu}\left(\frac{x_1}{\sqrt{V(t_1)}},\dots,\frac{x_r}{\sqrt{V(t_r)}}\right)=e^{-i\langle \bB(\bt),\bx\rangle}\exp\left(-G(x_1,\dots x_r)\right) \, . \]
\end{proof}

The following proposition gives a bound on the function $F_\bt$.
\begin{prop}\label{Ft-bound}
    With the same notations as in Proposition~\ref{TCL}, we have uniformly for $\|\bx\|\leq \sqrt{V}/(4r\|\widehat{\bt}\|_\infty)\, ,$
    \[
    F_\bt (x_1,\dots,x_r)=1+O\left(\frac{r^3\|\widehat{\bt}\|_\infty^2\|\bx\|^4}{V}\right)
    \]
\end{prop}
\begin{proof}
    We use the inequality $$1-\exp\left(-\sum_{n\geq2}b_{2n}(\bx;\bt)\right)\leq \sum_{n\geq2}b_{2n}(\bx;\bt)\, . $$
    We now prove that for all $n\geq2$ $$b_{2n}(\bx;\bt)\leq 2^{2n}a_{2n}r^3\frac{\|\widehat{\bt}\|_\infty\|\bx\|^4}{V}\, .$$
    Using the Cauchy-Schwarz inequality we obtain
    \begin{align*}
        b_{2n}(\bx;\bt)&=2^{2n}a_{2n}\sum_{\chi\ne1}\sum_{\gamma_\chi>0}\frac{\left|\sum_{i=1}^r \langle t_i,\chi\rangle x_i V(t_i)^{-1/2}\right|^{2n}}{(\frac{1}{4}+\gamma_\chi^2)^n}\\
        &\le 2^{2n}a_{2n}r\sum_{\chi\ne1}\sum_{\gamma_\chi>0}\frac{\left(\sum_{i=1}^r |\langle t_i,\chi\rangle x_i|^2 V(t_i)^{-1}\right)}{\frac{1}{4}+\gamma_\chi^2}\cdot\frac{\left|\sum_{i=1}^r \langle t_i,\chi\rangle 2 x_i V(t_i)^{-1/2}\right|^{2(n-1)}}{(1+4\gamma_\chi^2)^{n-1}}\, .
    \end{align*}
    Our assumption on $\bx$ implies that \[\left|\sum_{i=1}^r \langle t_i,\chi\rangle 2 x_i V(t_i)^{-1/2}\right| \leq 1 \, . \]
    We deduce that
    \[ \frac{\left|\sum_{i=1}^r \langle t_i,\chi\rangle 2 x_i V(t_i)^{-1/2}\right|^{2(n-1)}}{(1+4\gamma_\chi^2)^{n-1}} \leq r\frac{\| \widehat{\bt} \|_\infty ^2 \|\bx\|^2}{V} \, .\]
    Thus, 
    \begin{align*}
        b_{2n}(\bx;\bt)&=2^{2n}a_{2n}\sum_{\chi\ne1}\sum_{\gamma_\chi>0}\frac{\left|\sum_{i=1}^r \langle t_i,\chi\rangle x_i V(t_i)^{-1/2}\right|^{2n}}{(\frac{1}{4}+\gamma_\chi^2)^n}\\
        &\le 2^{2n}a_{2n}r^2 \frac{\| \widehat{\bt} \|_\infty ^2 \|\bx\|^2}{V}\sum_{i=1}^r\sum_{\chi\ne1}\sum_{\gamma_\chi>0}\frac{\left( |\langle t_i,\chi\rangle x_i|^2 V(t_i)^{-1}\right)}{\frac{1}{4}+\gamma_\chi^2}\\
        &=2^{2n}a_{2n}r^2 \frac{\| \widehat{\bt} \|_\infty ^2 \|\bx\|^2}{V}\cdot\frac{x_1^2+\dots+x_r^2}{2}\le 2^{2n}a_{2n}r^3 \frac{\| \widehat{\bt} \|_\infty ^2 \|\bx\|^4}{V}\, .
    \end{align*}
    This proves the desired inequality. It suffices to sum over $n\geq 2$ and use $a_{2n}\ll (5/12)^{2n}$ to deduce the Proposition.
\end{proof}

\subsection{General explicit formula}\label{subsec:explicitformula}

In this subsection, we fix a Galois extension $L/\Q$ with group $G$ and for which $\mathsf{GRH}$, $\mathsf{AC}$, and $\mathsf{LI}$$^{-}$ hold. Let $\bt=(t_1,\dots,t_r)$ be a vector of linearly independent class functions on $G$ satisfying $\langle t_i,1\rangle=0$ for $1\le i \le r$.
\begin{lem}\label{lem : Delta(t)-pos-def}
    The matrix $\Delta(\bt)$ is positive-definite. Thus, $\vp_\bt$ (defined in~\eqref{def : phibt}) is the characteristic function of a non-degenerate normal centered random vector with density $f_\bt$ (defined in~\eqref{def : fbt}).
\end{lem}
\begin{proof}
    Let $\bx=(x_1,\dots,x_r)\in \R^r$. We have 
\[
\bx ^T\, \Delta(\bt)\, \bx=2 \left(\frac{x_1^2+\dots+x_r^2}{2}+\sum_{1\leq i<j \leq r} \rho(t_i,t_j)x_ix_j \right)=2b_{2}(\bx;\bt) \geq 0 \, ,
\]
where $b_2(\bx;\bt)$ is defined in Proposition~\ref{TCL}.
    We now prove that $\Delta(\bt)$ is invertible by proving that its column vectors are linearly independent over $\R$. Let $a_1,\dots,a_r \in \R$ satisfying
    \[
    \sum_{j=1}^ra_j \rho(t_i,t_j)=0\qquad(1\leq i \leq r)\, .
    \]
    This condition can be rewritten as 
    \[
   \sum_{\chi \ne 1} \Re\left(\left\langle \frac{t_i}{\sqrt{V(t_i)}},\chi \right\rangle  \overline{\left \langle \sum_{j=1}^r \frac{a_j t_j}{\sqrt{V(t_j)}},\chi \right\rangle} \sum_{\gamma_\chi>0}\frac{1}{\frac{1}{4}+\gamma_\chi^2} \right)=0\qquad(1\leq i \leq r)\, .
    \]
    Define
    \[
    S_i:= \sum_{\chi \ne 1} \Re\left(\left\langle \frac{t_i}{\sqrt{V(t_i)}},\chi \right\rangle  \overline{\left \langle \sum_{j=1}^r \frac{a_j t_j}{\sqrt{V(t_j)}},\chi \right\rangle} \sum_{\gamma_\chi>0}\frac{1}{\frac{1}{4}+\gamma_\chi^2} \right)\qquad(1\leq i \leq r)\, ,
    \]
    Therefore,
    \[0=\sum_{i=1}^r a_i S_i=\sum_{\chi\ne 1} \left|\left\langle \sum_{i=1}^r \frac{a_i t_i}{\sqrt{V(t_i)}},\chi \right \rangle \right |^2\sum_{\gamma_\chi>0}\frac{1}{\frac{1}{4}+\gamma_\chi^2}
    \]
    Thus, for all $\chi\in \Irr(G)$ we have 
    $$\left \langle \sum_{i=1}^r \frac{a_i}{\sqrt{V(t_i)}}t_i , \chi \right \rangle =0\, .$$
    Hence $$\sum_{i=1}^r \frac{a_i}{\sqrt{V(t_i)}}t_i=0\, .$$
    Since the family $(t_1,\dots, t_r)$ is linearly independent over $\R$, we conclude that for all $1\leq i \leq r$ we have $a_i=0$.
\end{proof}
Let $X_\bt=(X_1,\dots,X_r)$ be a random vector with values in $\R^r$ with distribution $\mu_\bt$, and define $$Y_\bt:=\left(\frac{X_1-E(t_1)}{\sqrt{V(t_1)}},\dots,\frac{X_r-E(t_r)}{\sqrt{V(t_r)}}\right)\, .$$
The characteristic function of $Y_\bt$ is given by 
\[
\widehat{Y}_\bt(\bx)=\vp_\bt(\bx) F_\bt(\bx)\qquad(\bx\in \R^r) \, ,
\]
where $F_\bt$ is defined in Proposition~\ref{TCL}.
The following Lemma contains the main analytic estimate needed for our explicit formula.
\begin{lem}\label{main-lem-expl}
    We have, uniformly for $\bx \in \R^r$ with $\|\bx\| \leq \sqrt{V}/(4r\|\widehat{\bt}\|_\infty)$ and $V=\min _{1\leq i \leq r} V(t_i)$: 
    $$\left|\Lambda \left(\widehat{Y}_\bt\right)(x_1,\dots,x_r)-\Lambda(\vp_\bt)(x_1,\dots,x_r)\right| \ll_r |x_1\dots x_r|\left(\|\bx\|^2+\|\bx\|^{2r} \right)M_\bt(\bx)\left (\frac{\|\widehat{\bt}\|_\infty^2}{V}+\left(\frac{\|\widehat{\bt}\|_\infty^2}{V} \right)^{2r}\right)\, ,$$
    where \[ M_\bt(\bx):=\int_0^1\dots\int_0^1 \exp\left(-\lambda_\bt\frac{s_1^2x_1^2+\dots+s_r^2x_r^2}{2} \right)\, \mathrm{d}s_1\dots\,  \mathrm{d}s_r\, , \]
    and $\lambda_\bt>0$ is the minimal eigenvalue of $\Delta(\bt)$.
\end{lem}
We first derive Theorem~\ref{explicit-formula} from Lemma~\ref{main-lem-expl}.
\begin{proof}[Proof of Theorem~\ref{explicit-formula}]
    We apply the Berry--Esseen inequality (Theorem~\ref{BerryEsseen}) with\\ $T=4r\sqrt{V}/\|\widehat{\bt}\|_\infty$, $Y=Y_\bt$ and $Z$ a centered reduced Gaussian random vector whose characteristic function is $\vp_\bt$. We note that
    \[
\delta^{(r+1)}_{L/\Q}(C_1,\dots,C_{r+1})=F_{X_\bt}(0,\dots,0)=F_Y(-B_1,\dots -B_r) \, , \]
and that \[F_Z(-B_1,\dots,-B_r)=\int_{-\infty}^{-B_1}\dots\int_{-\infty}^{-B_r}f_\bt(x_1,\dots,x_r)\,\mathrm{d} \bx\, .\]
It suffices to bound the error term. By Lemma~\ref{main-lem-expl}, it suffices to bound the following term
    \begin{align*}
        &\int_{\|\bx\|\le T}M_\bt(\bx)\left(\|\bx\|^2+\|\bx\|^{2r}\right) \, \mathrm{d}\bx \\
        &=\int_0^1\dots\int_0^1 \int_{\|\bx\|\le T}\left(\|\bx\|^2+\|\bx\|^{2r}\right)\exp\left(-\lambda_\bt\frac{s_1^2x_1^2+\dots+s_r^2x_r^2}{2} \right)\, \mathrm{d}\bx\, \mathrm{d}s_1\dots\,  \mathrm{d}s_r\\
        &=\int_0^1\dots\int_0^1 \int_{\|\by\|\le \sqrt{\lambda_\bt} T}\left(\frac{\|\by\|^2}{\lambda_\bt}+\frac{\|\by\|^{2r}}{\lambda_\bt^r}\right)\exp\left(-\frac{s_1^2y_1^2+\dots+s_r^2y_r^2}{2} \right)\, \mathrm{d}\by\, \mathrm{d}s_1\dots\,  \mathrm{d}s_r\\
        &\leq \left(\frac{1}{\lambda_\bt}+\frac{1}{\lambda_\bt^r}\right)\int_0^1\dots\int_0^1 \int_{\R^r}\left(\|\bx\|^2+\|\bx\|^{2r}\right)\exp\left(-\frac{s_1^2x_1^2+\dots+s_r^2x_r^2}{2} \right)\, \mathrm{d}\bx\, \mathrm{d}s_1\dots\,  \mathrm{d}s_r\\
        &\ll_r \left(\frac{1}{\lambda_\bt}+\frac{1}{\lambda_\bt^r}\right)\, .
    \end{align*}
and since the other terms in the error term of the Berry--Esseen formula are of lower dimensions, one can deduce the general result by induction on $r$.
    
\end{proof}
The remainder of this section is devoted to the proof of Lemma~\ref{main-lem-expl}. The proof of the following combinatorial Lemma is postponed to the Appendix.
\begin{lem}\label{PA}
    For $\ba=(a_1,\dots,a_r) \in \C^r$ and $\bx=(x_1,\dots,x_r)\in \R^r $ define 
    \[
    f(\ba,\bx)=\sum_{i=1}^r a_i x_i \, .
    \]
    For all $A\subset L:=\{\,1,\dots,r\, \}$ with cardinality $\alpha\le 2n$, there exists $P_A(\bS,\bT,\bX)$ a polynomial in the indeterminates $\bS=(S_1,\dots,S_r)$, $\bT=(T_1,\dots,T_r)$ and $\bX=(X_1,\dots,X_r)$ such that for all $(\ba,\bx) \in \C^r \times \R^r$, we have 
    \begin{align*}
        \frac{\partial ^\alpha |f|^{2n}}{\prod_{i\in A}\partial x_i}(\ba,\bx)&=P_A(\ba,\overline{\ba},\bx)|f(\ba,\bx)|^{2(n-\alpha)}\qquad(\alpha\le n)\, ,\\
        \frac{\partial ^\alpha |f|^{2n}}{\prod_{i\in A}\partial x_i}(\ba,\bx)&=P_A(\ba,\overline{\ba},\bx)\qquad \qquad \qquad (n\le \alpha \le 2n)\, .
    \end{align*}
    Moreover, $P_A(S,T,X)$ satisfies the following properties:
    \begin{enumerate}
        \item The polynomial $P_A$ is homogeneous (separately) in each tuple of indeterminates $\bS,\, \bT$ and $\bX$ and has non-negative integer coefficients.
        \item $P_A(\bS,\bT,\bX)$ has degree $\alpha$ in each tuple of indeterminates $\bS,\, \bT$ and $\bX$ when $\alpha\le n$ and degree $n$ in $\bS$ and $\bT$ and has degree $2n-\alpha$ in $\bX$ when $n\le \alpha \le 2n$.
        \item The coefficients of $P_A$ are $\le n^\alpha$ when $\alpha \le n$ and $\ll_r 1$ when $n<\alpha\le 2n$. 
    \end{enumerate}
    
\end{lem}
We now use Lemma~\ref{PA} to estimate the partial derivatives of $x\mapsto \sum_{n\geq2}b_{2n}(\bx;\bt) $.
\begin{lem}\label{estpartderiv}
     For all $\bx\in \R^r$ such that $\|\bx\|\leq \sqrt{V}/(4r\|\widehat{\bt}\|_\infty)=:R$, define $S(\bx;\bt):=\sum_{n\geq2}b_{2n}(\bx;\bt)$. For all $A\subset L:=\{1,\dots,r\}$ with cardinality $\alpha \geq 1$, the map $S(\, \cdot\,  ;\bt)\,:\,\bx\mapsto S(\bx;\bt)$ admits the partial derivative $\partial ^\alpha S(\, \cdot\, ;\bt)/\prod_{i\in A}\partial x_i$ which is continuous on the ball $B(0;R)$ and satisfies uniformly for $\bx\in B(0,R)$: 
     \[
     \frac{\partial ^\alpha S(\, \cdot\, ;\bt)}{\prod_{i\in A}\partial x_i}(\bx)\ll_r \left(\|\bx\|^3+\|\bx\|^{\alpha}\right)\left (\frac{\|\widehat{\bt}\|_\infty^2}{V}+\iota_\alpha\left(\frac{\|\widehat{\bt}\|_\infty^2}{V} \right)^{\alpha-1}\right)  \, ,
     \]
     where $\iota_\alpha=0$ if $\alpha=1$.
\end{lem}
\begin{proof}
    It suffices to prove that for all $A\subset L$ with cardinality $\alpha$, the function series\\
    $\sum_{n\geq2} \partial^\alpha b_{2n}(\, \cdot\, ,\bt)/\bigl(\prod_{i\in A}\partial x_i \bigr)$
    is uniformly convergent on the ball $B(0,R)$ and satisfies uniformly for $\bx\in B(0,R)$
    $$\sum_{n\geq2} \left| \frac{\partial^\alpha b_{2n}(\, \bx \, ,\bt)} {\prod_{i\in A}\partial x_i } \right|\ll_r \left(\|\bx\|^2+\|\bx\|^{r+2}\right)\left (\frac{\|\widehat{\bt}\|_\infty^2}{V}+\iota_\alpha \left(\frac{\|\widehat{\bt}\|_\infty^2}{V} \right)^{\alpha-1}\right)  \, .$$
    Fix $A\subset L$, and define for all $i\in L$ and $\chi \in \Irr(G)$ and all $n\geq 1$
    $$a_i(\chi):=\frac{\langle t_i,\chi\rangle}{\sqrt{V(t_i)}},\quad B_n(\chi):=\sum_{\gamma_\chi>0}\frac{1}{(\frac{1}{4}+\gamma_\chi^2)^n },\quad f_\chi(\bx):=\sum_{j=1}^ra_j(\chi) x_j\, .$$
    We have for all $n\geq 2$, 
   \(b_{2n}(\bx;\bt)=2^{2n}a_{2n}\sum_{\chi\ne1}B_n(\chi) \left| f_\chi (x)\right|^{2n}.
   \)
   Applying Lemma~\ref{PA}, we deduce that for all $\chi\ne1$ and for all $n\geq \alpha$ there exist polynomials $P_{A,\chi}(\bS,\bT,\bX)$ such that
   \begin{equation}
   \left|\frac{\partial ^\alpha b_{2n}(\bx;\bt)}{\prod_{i\in A}\partial x_i}\right|=2^{2n}a_{2n} \left |\sum_{\chi \ne 1}  B_n(\chi) P_{A,\chi}\left(\ba (\chi), \overline{\ba (\chi)}, \bx\right) \left |f_{\chi}(\bx)\right|^{2(n-\alpha)} \right|\, ,
   \end{equation}
   where $\ba(\chi)=(a_1(\chi),\dots,a_r(\chi))$. 
   Applying properties 2. and 3. from Lemma~\ref{PA} of polynomials $P_{A,\chi}$ we obtain that for all $n\geq \alpha$
   \begin{equation}\label{partderiv}
   \left|\frac{\partial ^\alpha b_{2n}(\bx;\bt)}{\prod_{i\in A}\partial x_i}\right|\ll_r 2^{2n}a_{2n} n^{\alpha}\sum_{\chi \ne1} B_n(\chi) \|a(\chi)\|_2^{2\alpha}\|\bx\|^\alpha\,  |f_\chi(\bx)|^{2(n-\alpha)}\, .
   \end{equation}
   Since for all $\bx\in B(0,R)$ and all $m\geq1$.
   \[ B_n(\chi) =2^{2n}\sum_{\gamma_\chi >0}\frac{1}{(1+4\gamma_\chi^2)^n}\leq 2^{2(n-1)} B_1(\chi)\ \text{ and }\ 2^{2m } |f_\chi(\bx) |^{2m}=|f_\chi(2\bx)|^{2m}\leq 1 \, , \]
   for $\alpha=1$, we have
   \[
   \left|\frac{\partial ^\alpha b_{2n}(\bx;\bt)}{\prod_{i\in A}\partial x_i}\right|\ll_r 2^{2n+1}a_{2n} n \|\bx\| \sum_{\chi \ne1} B_1(\chi) \|a(\chi)\|_2^{2}\,  |f_\chi(\bx)|^{2}\, .
   \]
   A triangle inequality shows that $|f_\chi(\bx)|\leq r \|\widehat{\bt}\|_\infty\|\bx\|/V^{1/2}\, .$
   Moreover, we have 
   \[
   \sum_{\chi \ne 1} B_1(\chi)\|a(\chi)\|_2^2=\sum_{i=1}^r \frac{1}{V(t_i)}\sum_{\chi \ne 1}|\langle t_i,\chi\rangle|^2B_1(\chi)=\frac{r}{2}\, .
   \]
   Thus, for $\alpha=1$,
   \[
   \left|\frac{\partial ^\alpha b_{2n}(\bx;\bt)}{\prod_{i\in A}\partial x_i}\right| \ll_r 2^{2n}a_{2n}n \|\bx\|^3 \frac{\|\widehat{\bt}\|_\infty^2}{V}
   \]
   Hence, for $\alpha=1$, $$ \sum_{n\geq2}\left| \frac{\partial^\alpha b_{2n}(\, \bx \, ,\bt)} {\prod_{i\in A}\partial x_i } \right| \ll_r \|\bx\|^3 \frac{\|\widehat{\bt}\|_\infty^2}{V} \, .   $$
   From now on we assume that $\alpha \geq 2$, using the inequality \eqref{partderiv} we obtain with same ideas as above:
   \[
   \left|\frac{\partial ^\alpha b_{2n}(\bx;\bt)}{\prod_{i\in A}\partial x_i}\right| \ll_r 2^{2(n+\alpha)}a_{2n}n^\alpha \|\bx\|^\alpha \frac{\|\widehat{\bt}\|_\infty^2}{V}\sum_{\chi \ne 1}B_1(\chi)\|a_\chi\|_2^2\ll_r 2^{2n}a_{2n}n^\alpha \|\bx\|^\alpha \frac{\|\widehat{\bt}\|_\infty^2}{V}
   \]
   Applying Lemma~\ref{PA}, we obtain also that for all $n\geq 2$ such that $n<\alpha\leq 2n$ and all $\chi\ne 1$ there exists a polynomial $P_{A,\chi}$ such that \begin{equation}
   \left|\frac{\partial ^\alpha b_{2n}(\bx;\bt)}{\prod_{i\in A}\partial x_i}\right|=2^{2n}a_{2n} \left |\sum_{\chi \ne 1}  B_n(\chi) P_{A,\chi}\left(\ba (\chi), \overline{\ba (\chi)}, \bx\right) \right|\, .
   \end{equation}
   Applying the properties of the polynomials $P_{A,\chi}$ when $n<\alpha\leq 2n$ we obtain 
   \[
   \sum_{2\leq n <\alpha} \left|\frac{\partial ^\alpha b_{2n}(\bx;\bt)}{\prod_{i\in A}\partial x_i}\right| \ll_r \|\bx\|^{2n-\alpha} \left(\frac{\|\widehat{\bt}\|_\infty^2}{V} \right)^{\alpha-1}
   \]
   Thus, for $\alpha\geq 2$ we have 
   $$ \sum_{n\geq2} \left|\frac{\partial^\alpha b_{2n}(\, \bx \, ,\bt)} {\prod_{i\in A}\partial x_i } \right| \ll_r \|\bx\|^\alpha \left (\frac{\|\widehat{\bt}\|_\infty^2}{V}+\left(\frac{\|\widehat{\bt}\|_\infty^2}{V} \right)^{\alpha-1}\right)  \, .   $$
\end{proof}

\begin{lem}\label{hX}
    Let $X$ be a partition of $L=\{1,\dots,r\}$. Define 
    \(h_X:=\prod_{J \in X}\widehat{Y}_\bt \circ \psi_{J,L}-\prod_{J\in X}\vp_\bt \circ \psi_{J,L}\, .
    \)
    We have 
    \[ \frac{\partial ^r h_X}{\partial x_1 \dots \partial x_r} (x_1,\dots, x_r)\ll_r \left(\|\bx\|^2+\|\bx\|^{2r} \right)\exp\left(-\lambda_\bt \frac{x_1^2+\dots+x_r^2}{2}\right)\left (\frac{\|\widehat{\bt}\|_\infty^2}{V}+\left(\frac{\|\widehat{\bt}\|_\infty^2}{V} \right)^{2r}\right)\, .
    \]
\end{lem}
\begin{proof}
Recall $F_\bt(\bx)=\exp(-S(\bx;\bt))$ and $\vp_\bt(\bx)=\exp\left(-\frac{x_1^2+\dots+ x_r^2}{2}-\sum_{i<j}\rho(t_i,t_j)x_ix_j \right)$. In order to study the partial derivatives of $F_\bt$ and $\vp_\bt$, we use a multivariate Faà di Bruno formula. Let $A\subset L$ with cardinality $\alpha \geq 1$. An application of Faà Di Bruno formula on $\vp_\bt$ (or even a simple induction on $\alpha$) shows that there exists a polynomial $Q_A$ of degree $\alpha$ with coefficients satisfying the bound $\ll_r 1$ and such that 
\begin{equation}\label{partialphi} \frac{\partial^\alpha \vp_\bt}{\prod_{i\in A} \partial x_i} (x_1,\dots,x_r)=Q_A(x_1,\dots,x_r)\vp_\bt(x_1,\dots,x_r) \, . \end{equation}
Applying again Faà Di Bruno formula we obtain
\begin{align} \frac{\partial ^\alpha F_\bt}{\prod_{i\in A}\partial x_i} (\bx)&=F_\bt (\bx) \sum_{\frak{P} \in \Pi_A} \prod_{B \in \frak{P}} \frac{\partial^{|B|}S(\bx;\bt)}{\prod_{k\in B} \partial x_k} \label{partialFt} \\
&\ll_r \left( \|\bx\|^2+\|\bx\|^{2\alpha}\right) \left( \frac{\| \widehat{\bt}\|_\infty^2}{V}+\left(\frac{\| \widehat{\bt}\|_\infty^2}{V} \right)^{2\alpha}\right) \nonumber\, ,
\end{align}
where we used Lemma~\ref{estpartderiv} on each $B\in \frak{P}$ and used $\alpha=\sum_{B\in \frak{P}}|B|$.\\
Since $\widehat{Y}_\bt=F_\bt \cdot \vp_\bt$ we have 
\[ h_X=\left(\prod_{J\in X} \vp_\bt \circ \psi_{J,L}\right)\cdot \left( \prod_{J\in X} F_\bt \circ \psi_{J,L}-1\right) \, . \]
A multivariate Leibniz formula applied to the latter product yields:

\begin{align*}
\frac{\partial ^r h_X}{\partial x_1 \dots \partial x_r}(\bx)&=\frac{\partial ^r \left(\prod_{J\in X} \vp_\bt \circ \psi_{J,L}\right)}{\partial x_1 \dots \partial x_r}(\bx)  \left( \prod_{J\in X} F_\bt \circ \psi_{J,L}(\bx)-1\right)\\
&+\sum_{\emptyset\ne K\subset L}\left(\frac{\partial ^{|K|}\prod_{J\in X}(F_\bt \circ \psi_{J,L})}{\prod_{j\in K }\partial x_j} \right)(\bx)\cdot \frac{\partial^{r-|K|}\left(\prod_{J\in X} \vp_\bt \circ \psi_{J,L}\right)}{\prod_{j\in L\setminus K} \partial x_j}(\bx)\, .
\end{align*}
Note that, for $j\in L$ and $J\in X$, only one factor of the product $\prod_{J\in X} \vp_\bt \circ \psi_{J,L}$ depends on $x_j$. (Note that the same remark holds when we replace $\vp_\bt$ by $F_\bt$)\\
Hence, by \eqref{partialphi}, for all $K\subset L$ (including the case $K=\emptyset$) there exists a polynomial $H_K$ of degree $r-|K|$ such that
\[ \frac{\partial^{r-|K|}\left(\prod_{J\in X} \vp_\bt \circ \psi_{J,L}\right)}{\prod_{j\in L\setminus K} \partial x_j}(\bx)=H_K(\bx) \prod_{J\in X} \vp_\bt \circ \psi_{J,L}\, . \]
Since $\bx^T \Delta(\bt )\bx \geq \lambda_\bt \|\bx\|^2$, we have 
\[
\prod_{J\in X}\vp_\bt \circ \psi_{J,L}(\bx) \leq \exp\left( -\lambda_\bt \frac{x_1^2+\dots+x_r^2}{2} \right) \, . \]
We deduce that for all $K\subset L$
\[ \frac{\partial^{r-|K|}\left(\prod_{J\in X} \vp_\bt \circ \psi_{J,L}\right)}{\prod_{j\in L\setminus K} \partial x_j}(\bx) \ll_r \left(1+\|\bx\|^{r-|K|} \right)\exp\left( -\lambda_\bt \frac{x_1^2+\dots+x_r^2}{2} \right) \, . \]
Using \eqref{partialFt} we deduce that for all $\emptyset \ne K\subset L$ 
\[\left(\frac{\partial ^{|K|}\prod_{J\in X}(F_\bt \circ \psi_{J,L})}{\prod_{j\in K }\partial x_j} \right)(\bx) \ll_r \left( \|\bx\|^2+\|\bx\|^{2|K|}\right) \left( \frac{\| \widehat{\bt}\|_\infty^2}{V}+\left(\frac{\| \widehat{\bt}\|_\infty^2}{V} \right)^{2|K|}\right)  \, .\]
The lemma follows by applying Proposition~\ref{Ft-bound}.
\end{proof}

We restate \cite{HK}*{Lemma 3.1} which will be useful for our purposes.

\begin{lem}\label{H-K}
    Let $h:\C^r\to \C$ such that $h(0)=1$. For all $\bx \in \C^r$, if one of the coordinates of $\bx$ is $0$, then 
    $$\Lambda(h)(\bx)=0\, .$$
\end{lem}

We are now ready to prove Lemma~\ref{main-lem-expl}
\begin{proof}[Proof of Lemma~\ref{main-lem-expl}]
    For $\bx\in \R^r $ such that $\|\bx\| \leq \sqrt{V}/(4r\|\widehat{\bt}\|_\infty)$, define 
    \[
    h(\bx):=\Lambda \left(\widehat{Y}_\bt\right)(x_1,\dots,x_r)-\Lambda(\vp_\bt)(x_1,\dots,x_r) \, .
    \]
    By Lemma~\ref{H-K} if one of the coordinates of $\bx$ is $0$ then $h(\bx)=0$. 
    Since $F_\bt(\bx)=\exp(-S(\bx;\bt))$ and $\widehat{Y_\bt}(\bx)=F_\bt(\bx) \vp_\bt(\bx)$, by Lemma~\ref{estpartderiv}, for all $A\subset L:=\{1,\dots,r\}$ with cardinality $\alpha$, the partial derivative 
    \(\frac{\partial^\alpha h}{\prod_{i\in A}\partial x_i}\) exists and is continuous. We can apply Hadamard's lemma which implies that for all $\bx=(x_1,\dots,x_r)$ such that $\|\bx\| \leq \sqrt{V}/(4r\|\widehat{\bt}\|_\infty)$ we have
    \[
    h(x_1,\dots,x_r)=x_1\dots x_r \int_0^1\dots\int_0^1 \frac{\partial^rh}{\partial x_1\dots \partial x_r}(s_1x_1,\dots,s_rx_r)\, \mathrm{d}s_1\dots\,  \mathrm{d}s_r\, .
    \]
    With notations as in Lemma~\ref{hX} we have (by definition of $\Lambda$) 
    \[
    h(x_1,\dots,x_r)=\sum_{X\in \Pi_L} \mu_X h_X (x_1,\dots,x_r) \, .
    \]
    Thus, by Lemma~\ref{hX}  \[ \frac{\partial ^r h}{\partial x_1 \dots \partial x_r} (x_1,\dots, x_r)\ll_r \left(\|\bx\|^2+\|\bx\|^{2r} \right)\exp\left(-\lambda_\bt \frac{x_1^2+\dots+x_r^2}{2}\right)\left (\frac{\|\widehat{\bt}\|_\infty^2}{V}+\left(\frac{\|\widehat{\bt}\|_\infty^2}{V} \right)^{2r}\right)\, .
    \]
    It now suffices to integrate and see that $\|(s_1x_1,\dots,s_rx_r)\| \leq \|\bx\|$ for $0\leq s_1,\dots,s_r\leq 1$.
\end{proof}

\subsection{Special cases}\label{subsec:consequences}
Let $L/\Q$ be a Galois extension of number fields with group $G$ and let $C_1,\dots,C_{r+1}\in G^{\sharp}$ pairwise distinct. Denote $t_i=t_{C_i,C_{i+1}}$, for $1\le i \le r$. We rewrite the explicit formula given by Theorem~\ref{explicit-formula} in different cases. When $r=1$ (two-way races) the matrix $\Delta(\bt)$ is of order $1$ so $\lambda_\bt=1$. We deduce that: 
\begin{cor}\label{r=1} Assume $\mathsf{GRH}$, $\mathsf{AC}$, and $\mathsf{LI}^-$. Then,
    \[
    \delta_{L/\Q}^{(2)}(C_1,C_2)=\frac{1}{\sqrt{2\pi}}\int_{-\infty}^{-B_1}\exp\left( -\frac{x^2}{2}\right)\mathrm{d}x+O\left(\frac{\|\widehat{\bt}\|_\infty}{\sqrt{V}}+\frac{\|\widehat{\bt}\|_\infty^4}{V^2} \right)  \, ,
    \]
    where $B_1=B_{L/\Q}(C_1,C_2)$.
    
\end{cor}
\begin{proof}
This is a direct consequence of Theorem~\ref{explicit-formula}, since the only possible eigenvalue is $1$.
\end{proof}
A first-order Taylor expansion recovers \cite{FJ}*{Theorem 5.10}. We note that having a different sign in our formula is normal since in our work $\delta_{L/\Q}(C_1,C_2)$ is the logarithmic density of $x\geq 2$ such that \[\frac{\pi(x;L/\Q;C_1)}{|C_1|}<\frac{\pi(x;L/\Q;C_2)}{|C_2|} \, ,\]
while in \cite{FJ} the logarithmic density is related to sets with opposite strict inequality.\\
For $r\ge 2$, in order to apply Theorem~\ref{explicit-formula} we need to bound the eigenvalue factor in the error term. We note that two types of terms occur in the matrix $\Delta(\bt)$, and we can both relate them to the function $U_{L/\Q}$ as follows: 
1) Terms that are adjacent to the diagonal have the general form 
 \begin{equation}\label{rhoii+1} \rho(t_{a,b},t_{b,c})=\frac{-1+U_{L/\Q}(ab^{-1})+U_{L/\Q}(bc^{-1})-U_{L/\Q}(ac^{-1})}{\sqrt{\bigl(2-2U_{L/\Q}(ab^{-1})\bigr)\bigl(2-2U_{L/\Q}(bc^{-1})\bigr)}} \, ,\end{equation}
 for some $(a,b,c)\in \mathcal{A}_3(G)$.\\
2) Terms with indices $(i,j)$ such that $|j-i|\ge 2$:
\begin{equation}\label{rhoij}
            \rho(t_{a,b},t_{c,d})=\frac{U_{L/\Q}(ac^{-1})+U_{L/\Q}(cd^{-1})-U_{L/\Q}(bc^{-1})-U_{L/\Q}(ad^{-1})}{\sqrt{\bigl(2-2U_{L/\Q}(ab^{-1})\bigr)\bigl(2-2U_{L/\Q}(cd^{-1})\bigr)}} \, ,
        \end{equation}  
for some $(a,b,c,d)\in\mathcal{A}_4(G) $.\\
The following Lemma shows that when $r=2$ the minimal eigenvalues can never be too small.
\begin{lem}\label{encadrementrho}
Let $L/\Q$ be an abelian extension with group $G$ for which $\mathsf{GRH}$ holds, and let $a,b,c \in \mathcal{A}_3(G)$. Then, 
\[ -\frac{3}{4}+o(1) \le \rho\bigl(t_{a,b},t_{b,c}\bigr) \le o(1) \quad (d_L\to \infty) \, .\]
\end{lem}
\begin{proof}
Denote $U=U_{L/\Q}$.
Since $-U(x)=|U(x)|+o(1)$ ($d_L\to \infty)$, then, for all $x\ne 1$ we have \[ o(1) \le -U(x)<1 \, .\]
Hence for all $x\ne 1$ we have $2-2U(x) \ge 2+o(1)$ and $-1+U(ab^{-1})+U(bc^{-1})-U(ac^{-1})<o(1)$. Thus~\eqref{rhoii+1} shows that $\rho\bigl(t_{a,b},t_{b,c}\bigr)\le o(1)$. We have also 
\[ \rho(t_{a,b},t_{b,c}) \geq \frac{-1+U(ab^{-1})+U(bc^{-1})}{\sqrt{(2-2U(ab^{-1}))(2-2U(bc^{-1}))}}+o(1) \, ,\]
A simple study of the function $f\colon [\, -1\, ,\ve\,]^2\to \R$ (for some small $\ve>0$), defined by 
\[f(x,y)=\frac{-1+x+y}{\sqrt{(2-2x)(2-2y)}} \, ,\]
shows that $f$ takes its minimum at $(x,y)=(-1,-1)$ and in this case $f(-1,-1)=-3/4$. This proves the Lemma.
\end{proof}
This proves that for any $a,b,c \in \mathcal{A}_3(G)$ the minimal eigenvalue of the matrix $\Delta_{L/\Q}^{(2)}(a,b,c)$ is \begin{equation}\label{3way-eigenvalue}
    \lambda_{\min}\bigl(\Delta_{L/\Q}^{(2)}(a,b,c) \bigr)=  1-\bigl|\rho\bigl(t_{a,b},t_{b,c}\bigr) \bigr|\ge \frac{1}{4}+o(1)\, .
\end{equation} 
The case $r=2$ can be stated as follows.
\begin{cor}\label{3wayraces}
     Assume $\mathsf{GRH}$, $\mathsf{AC}$, and $\mathsf{LI}^-$.
     If $B(t_1)$ and $B(t_2)$ are sufficiently close to $0$, then
    \begin{align*}
        \delta_{L/\Q}^{(3)}(C_1,C_2,C_3)&=\frac{1}{4}+\frac{1}{2\pi}\arcsin\left(\rho(t_1,t_2)\right)-\frac{1}{2\sqrt{2\pi}}(B(t_1)+B(t_2))\\
        &+O\left(B(t_1)^2+B(t_2)^2+\left(1+\frac{1}{\lambda_\bt}+\frac{1}{\lambda_\bt^2}\right)\left(\frac{\|\widehat{\bt}\|_\infty}{\sqrt{V}}+\frac{\|\widehat{\bt}\|_\infty^8}{V^4} \right)\right)\, .
    \end{align*}
    If, moreover, $\rho(t_1,t_2)$ is sufficiently close to $-1/2$, then
    \begin{align*}
   \delta_{L/\Q}^{(3)}&(C_1,C_2,C_3)=\frac{1}{6}-\frac{1}{\pi\sqrt{3}}\left(\rho(t_1,t_2)+\frac{1}{2}\right)-\frac{1}{2\sqrt{2\pi}}(B(t_1)+B(t_2))\\
    &+O\left((\rho(t_1,t_2)+1/2)^2+B(t_1)^2+B(t_2)^2+\left(1+\frac{1}{\lambda_\bt}+\frac{1}{\lambda_\bt^2}\right)\left(\frac{\|\widehat{\bt}\|_\infty}{\sqrt{V}}+\frac{\|\widehat{\bt}\|_\infty^8}{V^4} \right)\right)\, .
    \end{align*}
\end{cor}
\begin{proof}
Set $\rho=\rho(t_1,t_2)$. Our goal is to give the Taylor expansion at $(0,0)$ of the function: 
$$F(x,y)=\int_{-\infty}^{x}\int_{-\infty}^{y} f_\bt(x,y) \, \mathrm{d}x\mathrm{d}y \, .$$
The value $F(0,0)$ is known, see for instance~\cite{Kotz}*{Equation (46.47)}, we have 
$$F(0,0)=\frac{1}{4}+\frac{1}{2\pi}\arcsin\left(\rho \right)\, .$$
Denoting $\Phi$ the cdf of a standard real Gaussian, a change of variable shows that: 
$$\frac{\partial F}{\partial x}(x_1,y_1)=\int_{-\infty}^{y_1} f_\bt(x_1,y)\mathrm{d}y=\frac{1}{\sqrt{2\pi}}\exp\left(-\frac{x_1^2}{2}\right) \Phi\left(\frac{y_1-\rho x_1}{\sqrt{1-\rho^2}}\right) \, .$$
Thus, by symmetry we have $$\frac{\partial F}{\partial x}(0,0)=\frac{1}{2\sqrt{2\pi}}\,\ \ \text{ and }\ \  \frac{\partial F}{\partial y}(0,0)=\frac{1}{2\sqrt{2\pi}} \, .$$
This proves the result.
\end{proof}
Corollary~\ref{abelian-3wayraces} is thus a direct consequence of \eqref{3way-eigenvalue} with Corollary~\ref{3wayraces}. We have the following important consequence: 
\begin{cor}
Let $(L_n)_n$ be a family of abelian extensions for which $\mathsf{GRH}$ and $\mathsf{LI}$ hold, with respective Galois groups $G_n$, and such that $|G_n|\ge 3$, $d_{L_n}\to \infty$ and $r(G_n)/\sqrt{\log d_{L_n}} \to 0$.
There exists $\eta>0$ such that, for every $n\ge 1$ and every $(a,b,c)\in \mathcal{A}_3(G_n)$, we have
\[ \delta_{L_n/\Q}^{(3)}(a,b,c)> \eta \, .\]
\end{cor}
\begin{proof}
Since for all $a\ne b \in G_n$ we have $B(t_{a,b}) \ll r(G_n)/\log(d_{L_n})^{1/2}$ with an implied absolute constant, it follows that for all $a,b,c \in \mathcal{A}_3(G_n)$ we have: 
\[ \delta_{L_n/\Q}^{(3)}(a,b,c)=\frac{1}{4}+\frac{1}{2\pi}\arcsin( \rho(t_{a,b},t_{b,c}))+o(1)\ge \frac{1}{4}-\frac{1}{2\pi}\arcsin\left(\frac{\sqrt{3}}{2}\right)+o(1) \quad (n\to \infty)\, .\]
Thus, \[ \delta_{L_n/\Q}^{(3)}(a,b,c) \ge \frac{1}{12}+o(1) \quad (n\to \infty)\, .\]
This completes the proof.
\end{proof}

\section{Pointwise convergence}\label{sec:pointwise}
\subsection{Pointwise convergence of the variance related functions}\label{subsec:pointwisecv}
Let $L/\Q$ be a Galois extension with group $G$. Define the normalized variance $T_{L/\Q}:\mathcal{A}_2(G)\longrightarrow \R$ as follows: \begin{equation}\label{normalized-variance}
    T_{L/\Q}=\frac{1}{N_{L}}V_{L/\Q}
\end{equation}
where $N_L$ is defined in~\eqref{def : NL}.
\begin{lem}\label{lem : pos-real}
Let $L$ be an abelian extension over $\Q$ with group $G$ for which $\mathsf{GRH}$ holds. Then, for all distinct $a, b \in G$ $$(2+o(1))N_L\leq V_{L/\Q}(a,b) \leq 4 N_{L} \quad (d_L\to \infty)\, .     $$
In particular, $V_{L/\Q}(a,b) \asymp \log d_L$, with implied absolute constants.
\end{lem}
\begin{proof}
    Since for all distinct $a, b\in G$, we have \(T_{L/\Q}(a,b)=2-2U_{L/\Q}\left(ab^{-1}\right) \), and since $|U_{L/\Q}|\le 1$, then \[T_{L/\Q} \le 4\, .\]
    We use Corollary~\ref{negativity of U(a)} combined with Lemma~\ref{negativity} to deduce that 
    \[ T_{L/\Q}\geq 2+o(1) \quad (d_L \to \infty) \, .\]
    The lemma follows by multiplying by $N_L$.
\end{proof}

\begin{prop}\label{var-ab}
Let $(L_n)_n$ be an \emph{increasing} family of abelian extensions over $\Q$ for which $\mathsf{GRH}$ holds, with respective groups $(G_n)$ and such that $[L_n : \Q]\to \infty$. Then, $(U_{L_n/\Q})_n$ converges pointwise to $0$ and the sequence of variances $(T_{L_n/\Q})_n$ converges pointwise to $2$. In particular, for all $(a,b) \in \mathcal{A}_2(G_1)$ and any choice of a sequence of lifts $\left(a^{(n)},b^{(n)} \right) \in \mathcal{A}_2(G_n)$
\[ V\left(a^{(n)},b^{(n)}\right) \sim 2\log d_{L_n}\quad (n\to \infty) \, .\]
\end{prop}
\begin{proof}
Since for all $n\ge 1$ and all $a\ne b\in G_n$, we have \[T_{L_n/\Q}(a,b)=2-2U_{L_n/\Q}\left(ab^{-1}\right) \, ,\]
it suffices to prove the pointwise convergence of $\bigl(U_{L_n/\Q}\bigr)_n$. Let $n_0\ge 1$ and $a \in G_{n_0}\setminus\{1\}$ and for all $n\ge n_0$, let $a^{(n)} \in G_n \setminus \{1\} $ be a lift of $a$ by the canonical map. By Corollary~\ref{negativity of U(a)} and Lemma~\ref{UbllG} we have 
\[ \left|U_{L_n/\Q}\left( a^{(n)} \right)\right| =\frac{1}{N_{L_n}} \left| \sum_{\chi \in \Irr(G_n)} \chi(a) \log A(\chi)\right|+o(1) \ll \frac{|G_n|}{N_{L_n}}\log d_{L_{n_0}} +o(1)\quad (n \to \infty) \, .\]
By Theorem~\ref{rtdiscbounded}, the right-hand side of the previous inequality goes to $0$ as $n\to \infty$. This proves the proposition.
\end{proof}
We define the matrix $\Gamma_r \in \mathcal{M}_r(\mathbb{R})$ by
\begin{equation}\label{Gamma-r}
\Gamma_r = 
\begin{pmatrix}
1 & -\frac{1}{2} & 0 & \cdots & 0 & 0 \\
-\frac{1}{2} & 1 & -\frac{1}{2} & \ddots & & \vdots \\
0 & -\frac{1}{2} & 1 & \ddots & \ddots & 0 \\
\vdots & \ddots & \ddots & \ddots & -\frac{1}{2} & 0 \\
0 & & \ddots & -\frac{1}{2} & 1 & -\frac{1}{2} \\
0 & \cdots & 0 & 0 & -\frac{1}{2} & 1
\end{pmatrix}.
\end{equation}
An important consequence of Proposition~\ref{var-ab} is the following pointwise convergence of our covariance matrices: 
\begin{prop}\label{matrix-conv}
    Let $(L_n)_n$ be an \emph{increasing} family of abelian Galois extensions over $\Q$ for which $\mathsf{GRH}$ holds. The matrix $\Delta_{L_n/\Q}^{(r)}$ converges pointwise to $\Gamma_r$.

\end{prop}
\begin{proof}
This is a trivial consequence of Proposition~\ref{var-ab} combined with~\eqref{rhoii+1} and~\eqref{rhoij}.
\end{proof}

\subsection{Criteria for $2$-moderacy}\label{subsec:2mod}
We begin the subsection by applying Lemma~\ref{contributionalla's} to deduce a useful consequence:
\begin{lem}\label{mvaluesofU}
Let $m\geq2$. Then, for all but finitely many abelian extensions $L/\Q$ with group $G$ that satisfy $\mathsf{GRH}$ and for which: 
\begin{itemize}
    \item There are at most $m$ non-trivial elements $a\in G$ such that \[ \left|U_{L/\Q}(a) \right| >\frac{1}{m}\, .\] 
    \item Moreover, whenever \( \left|U_{L/\Q}(a) \right| >\frac{1}{m} \), one has $\text{Ord}(a) \le m+2$.
\end{itemize}
In particular, there exists a subgroup $H_m$ of order at most $m^{m+2}$ such that for all $a\in G\setminus H_m$ we have $\left|U_{L/\Q}(a) \right| \leq \frac{1}{m}$.
\end{lem}
\begin{proof}
Assume by contradiction that there exist $a_1,\dots,a_{m+1} \in G\setminus \{1\}$ pairwise distinct such that
\[ \left|U_{L/\Q}(a_i)\right|>\frac{1}{m}\quad(1\le i \le m+1) \, .\]
Since \[\sum_{i=1}^{m+1} \left| U_{L/\Q}(a_i) \right| =-\frac{1}{N_L}\sum_{i=1}^{m+1} \sum_{\chi\ne 1} \chi(a_i) \log A(\chi)+o(1) \, . \]
Consequently, \[\sum_{i=1}^{m+1} \left| U_{L/\Q}(a_i) \right| \le -\frac{1}{N_L}\sum_{x\ne 1}\sum_{\chi \ne 1}\chi(x)\log A(\chi) +o(1)=\frac{\log d_L}{N_L}+o(1)\, .\]
Hence \[1+\frac{1}{m} \le 1+o(1) \quad (d_L\to \infty) \, ,\]
which is impossible.\\
For the second part of the lemma, the idea is to see that if $x\in \langle a \rangle$ then 
\[ \left| \sum_{\chi} \chi(a) \log A(\chi) \right| \le \left | \sum_{\chi} \chi(x) \log A(\chi) \right| \, .\]
Indeed, by Lemma~\ref{negativity} it suffices to prove that for all $p$
\[ \# \{ i\geq 0\ :\ a\in G_i(p)\} \le \# \{ i\geq 0\ :\ x\in G_i(p)\}\, ,\]
which is trivial, since there is an inclusion of the above sets (that is because $x\in \langle a \rangle \subset G_i(p)$ whenever $a\in G_i(p)$). We deduce that 
\[\bigl|U_{L/\Q}(a) \bigr| \le  \bigl|U_{L/\Q}(x)\bigr|+o(1)\quad (d_L\to \infty) \, .\]
Thus, for $d_L$ large (in terms of $m$), if \( \left|U_{L/\Q}(a) \right| >1/m \), then for all $x\in \langle a\rangle \setminus \{1\}$ we have \( \left|U_{L/\Q}(x) \right| >1/(m+1)\, ,\) and by the first part of the lemma there are at most $m+1$ such elements, thus $\text{Ord}(a)\le m+2\, .$
\end{proof}
We now use Lemma~\ref{mvaluesofU} to prove one of the main tools to construct sequences of conjugacy classes that enable us to control covariance matrices:
\begin{lem}\label{limU(x)}
Let $(L_n)_n$ be a family of abelian extensions over $\Q$ for which $\mathsf{GRH}$ holds, with respective Galois groups $(G_n)_n$, such that $|G_n|\to \infty$. For each $n\geq1$, let $x_n \in G_n$ such that either for all $n\ge 1$ we have $x_n\ne 1$ and $$\lim_{n\to \infty}\left|U_{L_n/\Q}(x_n)\right|=\ell\in (0,1]\, ,$$
or for all $n\ge 1$, we have $x_n=1$.
Then, for all $r\geq2$ and all sufficiently large $n$ there exist $a_{1,n},\dots,a_{r,n} \in G_n\setminus \{1,x_n\}$ satisfying 
\begin{align*}
    &\lim_{n\to \infty} U_{L_n/\Q}\left(x_n a_{i,n}^{-1}\right)=0,\quad\lim_{n\to \infty}\  U_{L_n/\Q}\left( a_{i,n}\right)\ =0,\\
    &\lim_{n\to \infty} U_{L_n/\Q}\left(a_{i,n}a_{j,n}^{-1}\right)=0,\quad\lim_{n\to \infty} U_{L_n/\Q}\left( x_na_{i,n}a_{j,n}^{-1} \right)=0\, .
\end{align*}
Moreover, if $r(G_n)\to \infty$, one can choose all $a_{i,n}$ to be non-squares.
\end{lem}
\begin{proof}
First assume that for every $n\ge 1$ $x_n\ne 1$. For $m \ge 2$, when $L/\Q$ is an abelian extension with group $G$, we say that $G$ satisfies the property $\mathcal{P}_m$ if \emph{``there exists a subgroup $H_m \leq G$ of order at most $m^{m+2}$ and index at least $r+2$, such that for all $a \in G \setminus H_m$, we have $\left| U_{L/\Q}(a) \right| < \frac{1}{m}$''}.\\
Let $m_0 \ge 2$ be such that $\ell>\frac{2}{m_0}$. By Lemma~\ref{mvaluesofU}, there exists $N\ge 1$, such that for all $n\ge N$, $G_n$ satisfies $\mathcal{P}_{m_0}$ and such that $$\left|U_{L_n/\Q}(x_n) \right| > \frac{\ell}{2}>\frac{1}{m_0} \, .$$ Thus, for all $n\ge N$, the set $\{\, m\ge 2\ :\ G_n\ \text{satisfies}\ \mathcal{P}_m\, \}$ is non-empty. Denote \[m_n:=\max \{\, m\ge 2\ :\ G_n\ \text{satisfies}\ \mathcal{P}_m\, \}\, .\]
By Lemma~\ref{mvaluesofU}, for all $M\ge m_0$, there exists $n_0$ such that for all $n\ge n_0$, $G_n$ satisfies $\mathcal{P}_M$. In particular, for all $n\ge n_0$, we have $m_n\ge M$. This proves that \[\lim_{n\to \infty} m_n=\infty\, .\]
Since for all $n\ge N$, $|G_n|/|H_{m_n}| \ge r+2$, we can consider $a_{1,n},\dots,a_{r,n} \in G_n$, so that $a_{1,n}H_{m_n},\dots,a_{r,n}H_{m_n}$ are pairwise distinct and all distinct from $H_{m_n}=x_nH_{m_n}$ ($x_n\in H_{m_n}$ since $|U_{L_n/\Q}(x_n)|>1/m_0\ge 1/m_n$).
Thus, for all $n\ge N$, \[ \max_{1\le i\le r} \left( \left|U_{L_n/\Q}(a_{i,n})\right|,\left|U_{L_n/\Q}(x_na_{i,n}^{-1})\right|,\max_{j\ne i}\left(\left|U_{L_n/\Q}(a_{i,n}a_{j,n}^{-1})\right|,\left|U_{L_n/\Q}(x_na_{i,n}a_{j,n}^{-1})\right|\right)\right)\, <\, \frac{1}{m_n} \, ,\]
which proves the result. \\
If $r(G_n) \to \infty$, for $N_1 \ge N$ large enough, we have $r(G_n)>r+2$. Since $r(G_n)$ is the number of square roots of $1$ in $G_n$, by the isomorphism theorem, we have $r(G_n)=|G_n|/|G_n^2|$, where $G_n^2$ is the set of squares of $G_n$. Let $n\ge N_1$, two cases arise: \\
Case 1: assume $H_{m_n} \subset G_n^2$. Choose $a_{1,n},\dots,a_{r,n} \in G_n$, so that $a_{1,n}G_n^2,\dots,a_{r,n}G_n^2$ are pairwise distinct and are all distinct from $G_n^2=x_nG_n^2$. In this case the same statement holds, and since every $a_{i,n} \notin G_n^2$; thus every $a_{i,n}$ is a non-square.\\
Case 2: assume there exists $b\in H_{m_n}\setminus G_n^2$, then for each $1\le i\le r$, if $a_{i,n}$ is a square, replace it by $ba_{i,n}\in a_{i,n}H_{m_n}$ which is a non-square.  The result follows.\\
If for every $n\ge 1$, $x_n=1$, the proof remains unchanged, since the existence of the elements $a_{i,n}$ does not depend on $x_n$.
\end{proof}
\begin{rk}
We require Lemma~\ref{limU(x)} in its full generality only for the proof of Theorem~\ref{main-density}, and consequently for its corollaries, Theorems~C and~D.  
For the remaining results it suffices to know that there exists a sequence $(a_n)$ with
\(U_{L_n/\Q}(a_n)\to 0\); this follows from Lemma~\ref{limU(x)} by taking $x_n=1$.
\end{rk}

We can now prove Theorem~A.
\begin{proof}[Proof of Theorem~A]
    Assume that $r(G_n)/\sqrt{\log d_{L_n}} \underset{n\to \infty}{\longrightarrow}0$, and let $n\geq 1$ and $C_1,C_2$ be two conjugacy classes in $G_n$. Denote $t=t_{C_1,C_2}$. We first note that \[ | B(t) |=\left| \frac{\langle t, r_{G_n} \rangle}{\sqrt{V(t)}} \right| \le A\frac{r(G_n)}{\sqrt{\log d_{L_n}}} \,  \]
    where $A>0$ is an absolute constant (that is because by Lemma~\ref{lem : pos-real} $V(t) \asymp \log d_{L_n}$). 
    Applying Corollary~\ref{r=1} we see that 
    \begin{align*}\left|\delta_{L_n/\Q}(C_1,C_2)-\frac{1}{2}\right|&=\left|\frac{1}{\sqrt{2\pi}}\int_0^{-B(t)} \exp(-x^2/2) \, \mathrm{d} x\right|+O(1/\sqrt{\log d_{L_n}}) \\
    &\le \frac{1}{\sqrt{2\pi}}\int_0^{A\,  r(G_n)/(\log d_{L_n})^{1/2}}\exp \left(-x^2/2\right)\,  \mathrm{d}x+O\left(\frac{1}{\sqrt{\log d_{L_n}}}\right)
    \end{align*}
    Since the last term does not depend on $C_1,C_2$ and tends to $0$ as $n\to \infty$, we conclude that $(L_n)$ is uniformly $2$-moderate.\\
    Assume that \[\lim_{n\to \infty}\frac{r(G_n)}{\sqrt{\log d_{L_n}}}=\ell \in \R_{>0}\, \cup\{\infty\} \, .\]
    In particular, $r(G_n)\to \infty$. Applying Lemma~\ref{limU(x)} with $x_n=1$, we deduce the existence of a sequence $(b_n)$ such that $b_n\in G_n$ and for all sufficiently large $n$, $b_n$ is a non-square in $G_n$ that satisfies \[\lim_{n\to \infty} U_{L_n/\Q}(b_n)=0\, .\]
    Thus, \[\lim_{n\to \infty}T_{L_n/\Q}(1,b_n)=2\, ,\]
    or equivalently $V_{L_n/\Q}(1,b_n)\sim2 \log d_{L_n}$, which implies that \[ B_n:=B_{L_n/\Q}(1,b_n) \underset{n\to \infty}{\sim}-\frac{r(G_n)}{\sqrt{2\log d_{L_n}}}\, .\]
    By Corollary~\ref{r=1}, we have
    \[ \delta_{L_n/\Q}^{(2)}(1,b_n)=\frac{1}{\sqrt{2\pi}}\int_{-\infty}^{-B_n}\exp\left(-\frac{x^2}{2}\right)\, \mathrm{d}x + O\left(\frac{1}{\sqrt{\log d_{L_n}}}\right)\, .\]
    Thus, \[ \lim_{n\to \infty}  \delta_{L_n/\Q}^{(2)}(1,b_n) =\frac{1}{\sqrt{2\pi}}\int_{-\infty}^{\ell/\sqrt{2}}\exp\left(-\frac{x^2}{2}\right)\, \mathrm{d}x \, .\]
    Assume that $(L_n)_n$ is increasing. Choose $n_0$ sufficiently large so that $G_{n_0}$ contains a non-square element $b_{n_0}$. For every $n\ge n_0$ let $b_n=b_{n_0}^{(n)}$ be its lift to $G_n$. Then, for all $n\ge n_0$, $b_n$ is still a non-square; moreover by Lemma~\ref{UbllG} we have $\lim_{n\to \infty} U_{L_n/\Q}(b_n)=0$. The remainder of the proof is unchanged.
    This completes the proof of Theorem~A.
\end{proof}

\begin{examples}
\begin{enumerate}
    \item Let $(L_n)_n=\left(\Q(\zeta_n)\right)_n$ be the family of cyclotomic extensions, then we recover that $(L_n)_n$ is uniformly $2$-moderate since it is well known that $r(G_n) \ll |G_n|^\ve$ for all $\ve>0$.
    \item Consider a sequence of prime numbers $(p_n)_n$ all congruent to $1 \pmod 4$ and chosen so that for every $n \ge 1$ we have 
    \[ \frac{\log p_1\dots p_n}{2^{n+1}}>2^{2n}\, . \]
    Consider $L_n=\Q\bigl(\sqrt{p_1},\dots,\sqrt{p_n}\bigr)$ and denote $G_n=\Gal(L_n/\Q)$.
    Thus, \[ \frac{r(G_n)}{\sqrt{\log d_{L_n}}}=\frac{|G_n|}{\sqrt{\log d_{L_n}}}=\left(\frac{2^{2n+1}}{2^n \log p_1\dots p_n}\right)^{1/2}<\frac{1}{2^n}\, .\]
    This proves that $(L_n)$ is $2$-moderate.
    \item Let $(p_n)_n$ be the increasing sequence of all odd primes, and define $L_n=\Q\bigl(\sqrt{p_1},\dots,\sqrt{p_n}\bigr)$. By the prime number theorem in arithmetic progressions, we have $r(G_n)/\sqrt{\log d_{L_n}}\underset{n\to \infty}{\longrightarrow}\infty$. This proves that $(L_n)_n$ is not uniformly $2$-moderate. In fact, by the explicit formula (Corollary~\ref{r=1}), we can deduce that 
    \[ \{1,0\}\subset \overline{\delta_{L_n/\Q}^{(2)}\bigl(\mathcal{A}_2(G_n)\bigr)}\, .\]
\end{enumerate}

\end{examples}

\subsection{Pointwise moderacy}\label{subsec:pointwisemod}
\begin{defi}
Let $r\geq 1$. Denote by $S_r^{++}$ the set of positive-definite symmetric matrices. If $\Gamma \in S_r^{++}$, denote by $f(\cdot; \Gamma)$ the density function of a centered Gaussian random vector with covariance matrix $\Gamma$, namely \[ f \left((x_1,\dots,x_r);\Gamma\right):= \frac{1}{((2\pi)^{r}\det(\Gamma))^{1/2}}\exp\left(-\frac{1}{2} \bx^T \Gamma^{-1} \bx \right)\, .\]
We define $F_r:\R^r\times S_r^{++}\to \R$ by 
\[ 
    F_r(\bx;\Gamma) :=\int_{(-\infty,x_1]\times\dots \times (-\infty,x_r]} f(\by;\Gamma) \mathrm{d}\by \, .
    \]
    
\end{defi}

\begin{rk}
By the dominated convergence theorem $F_r$ is continuous on $\R^r\times S_r^{++}$.
\end{rk}
\begin{lem}\label{cdf(0)}
    For all $r\geq 1$ we have
    $$F_r\bigl((0,\dots,0);\Gamma_r \bigr)=\frac{1}{(r+1)!}\, ,$$
    where $\Gamma_r$ is defined in~\eqref{Gamma-r}.
\end{lem}
\begin{proof}
    Let $Y=(Y_1,\dots,Y_{r+1})$ be an $r+1$-dimensional standard Gaussian in $\R^{r+1}$. Define 
    \[
    X_i:=\frac{1}{\sqrt{2}}(Y_i-Y_{i+1})\quad (1\leq i \leq r) \, .
    \]
    We have for $1\leq i < r$ and $j$ with $|j-i|\geq 2$ \begin{align*}
        &\text{Cov}(X_i,X_i)=\frac{1}{2}\text{Cov}(Y_i-Y_{i+1},Y_i-Y_{i+1})=1\, ,\\
        &\text{Cov}(X_i,X_{i+1})=-\frac{1}{2}\, ,\text{ and }\  \text{Cov}(X_i,X_j)=0\, .
    \end{align*}   
   Hence $X:=(X_1,\dots,X_r)$ is a centered Gaussian vector whose covariance matrix is $\Gamma_r$. Thus, 
    \[ F_r\bigl((0,\dots,0);\Gamma_r \bigr)=P(X_1\leq 0,\dots,X_r\leq 0)=P(Y_1\leq Y_2\leq \dots \leq Y_{r+1})\, .
    \]
    By symmetry, we have  $$1=\sum_{\sigma\in \frak{S}_{r+1}} P\left(Y_{\sigma(1)}\leq \dots \leq Y_{\sigma(r+1)}\right)=(r+1)!F_r\bigl((0,\dots,0);\Gamma_r\bigr)\, . $$
    Thus $$F_r\bigl((0,\dots,0);\Gamma_r \bigr)=\frac{1}{(r+1)!}\, .$$
\end{proof}

\begin{lem}\label{r+1mod-rmod}
    Let $(L_n)_n$ be a family of finite Galois extensions satisfying $\mathsf{GRH}$ and $\mathsf{LI}^-$ and let $r \geq 2$. If $(L_n)_n$ is uniformly $(r+1)$-moderate then it is uniformly $r$-moderate. If moreover $(L_n)_n$ is increasing and pointwise $(r+1)$-moderate then it is pointwise $r$-moderate.
\end{lem}

\begin{proof}
    We only prove the case of uniform moderacy as the method is similar for pointwise moderacy. Let $(C_{1, n}, \dots, C_{r,n}) \in \mathcal{A}_r(G_n)$. Then, assuming that $|G_n^{\#}| \ge r+1$, for any $C \in G_n^{\#}$ distinct from $C_{1, n}, \dots, C_{r, n}$ we have \begin{align*}\mathcal{P}_{L_n/\Q}(C_{1,n}, \dots, C_{r,n}) = &\mathcal{P}_{L_n/\Q}(C, C_{1, n}, \dots, C_{r,n}) \cup \mathcal{P}_{L_n/\Q}(C_{1, n}, C, C_{2, n}, \dots, C_{r,n}) \cup \dots\\ &\cup \mathcal{P}_{L_n/\Q}(C_{1, n}, \dots, C_{r,n}, C) \cup \mathcal{T}\end{align*} where $\mathcal{T}$ is the set of $x \ge 2$ satisfying an equality of the form $\frac{\pi(x; L_n/\Q, C_{i,n})}{|C_{i,n}|} = \frac{\pi(x; L_n/\Q, C)}{|C|}$ and the corresponding strict ordering between the other counting functions. Clearly, the above union is disjoint. Moreover, it is known that under $\mathsf{LI}^-$, $\mathcal{T}$ has logarithmic density $0$ (see for instance~\cite{Dev}*{Theorem 2.1}). Therefore, we obtain \begin{align*}\delta_{L_n/\Q}^{(r)}(C_{1,n}, \dots, C_{r,n}) = &\delta^{(r+1)}_{L_n/\Q}(C, C_{1, n}, \dots, C_{r,n}) + \delta^{(r+1)}_{L_n/\Q}(C_{1, n}, C, C_{2, n}, \dots, C_{r,n}) + \dots\\ &+ \delta^{(r+1)}_{L_n/\Q}(C_{1, n}, \dots, C_{r,n}, C).\end{align*}
    If we assume that $\delta^{(r+1)}_{L_n/\Q}$ converges uniformly to $\frac{1}{(r+1)!}$, then this implies that $\delta^{(r)}_{L_n/\Q}$ converges uniformly to $\frac{1}{r!}$, which proves the claim.
\end{proof}

We now prove Theorem~\ref{thm : pointwisemoderacy}.
\begin{proof} [Proof of Theorem~\ref{thm : pointwisemoderacy}]
    The equivalence (1)$\Longleftrightarrow$(2) follows from Theorem~A.
    \\
    (1)$\implies$(3): Since uniform $2$-moderacy implies pointwise $2$-moderacy, we can assume that $r\geq 2$. Consider pairwise distinct elements $a_1,\dots,a_{r+1} \in G_{n_0}$, for some $n_0\geq 1$. For all $n\geq n_0$, let $(a_i^{(n)})_{1\leq i\leq r+1}$ be the respective lifts of $(a_i)_{1\leq i\leq r+1}$ in $G_n$. Define 
    \[
    t_i^{(n)}=|G|\left( \mathds{1}_{\{a_i^{(n)}\}}-\mathds{1}_{\{a_{i+1}^{(n)}\}} \right) \quad 1\leq i \leq r
    \]
    and $\bt_n:=\left(t_1^{(n)},\dots,t_r^{(n)}\right)$. By Proposition~\ref{matrix-conv} we have \[ \lim_{n\to \infty} \Gamma(\bt_n) =\Gamma_r \, . \]
    By Proposition~\ref{matrix-conv} and Theorem~\ref{explicit-formula} we have \[ \delta^{(r+1)}_{L_n/\Q}\left(a_1^{(n)},\dots,a_{r+1}^{(n)}\right)=F_r\left( \left(-B\bigl(t_1^{(n)}\bigr),\dots,-B\left(t_r^{(n)}\right)\right); \Gamma\bigl(\bt_n\bigr) \right)+O\left(\frac{1}{\sqrt{\log d_{L_n}}}\right) \, .  \]
    By assumption, for all $1\leq i \leq r$, $B\left(t_i^{(n)}\right) \underset{n\to \infty}{\longrightarrow}0$. Thus, by Lemma~\ref{cdf(0)}
    \[\lim_{n\to \infty} \delta_{L_n/\Q}\left(a_1^{(n)},\dots,a_{r+1}^{(n)}\right)=F_r\bigl((0,\dots,0);\Gamma_r\bigr)=\frac{1}{(r+1)!} \, . \]
    (3)$\implies$(1) By Lemma~\ref{r+1mod-rmod}, it is sufficient to prove that pointwise $2$-moderacy implies (1). We proceed by contrapositive. Up to considering a subsequence we may assume that $\displaystyle \lim_{n \to \infty} r(G_n)/(\log d_{L_n})^{1/2}>0$. By Theorem~A we deduce that $(L_n)$ is not pontwise $2$-moderate.
    
\end{proof}
\begin{exe}
    Let $G\simeq \Z/n_1\Z \times\dots \times \Z/n_s\Z$ be an abelian group. One can construct a Galois extension $L/\Q$ with group isomorphic to $G$ and satisfying $$\frac{\log d_L}{|G|}>|G|^3\, .$$
    This is classical; for each $1\le i\le s$ consider $p_i\equiv 1 \pmod{n_i}$ large enough so that $\log p_i>|G|^4$, and we choose these primes to be pairwise distinct (since there are infinitely many). $L$ can be taken as the compositum of extensions $K_i$, where $K_i\subset \Q(\zeta_{p_i})$ has Galois group $\Z/n_i\Z$ over $\Q$. If $L$ is constructed this way, then $\log d_L\ge \log p_1\ge |G|^4$. Thus,
    \[\frac{r(G)}{\sqrt{\log d_L}}\le \frac{|G|}{\sqrt{|G|^4}}= \frac{1}{|G|}\, .\]
    Thus, for any family $(G_n)_n$ of abelian groups such that $|G_n|\to \infty$, one can construct a family $(L_n)_n$, $n \ge 1$, of abelian extensions over $\Q$ with respective Galois groups isomorphic to $G_n$ and satisfying $\log d_{L_n}>|G_n|^4$. This yields further instances of families that are uniformly $2$-moderate. Equivalently, it is pointwise $r$-moderate for every $r\ge 2$. 
\end{exe}

\section{Uniform convergence}\label{sec:uniform}
\subsection{Uniform convergence of the variance related functions}\label{subsec:uniformcv}
Unlike the pointwise convergence, the uniform convergence of our functions (e.g. $U_{L_n/\Q}, V_{L_n/\Q},\Delta_{L_n/\Q}\dots$) is not guaranteed. Our goal is to understand how the uniform convergence of these functions relates to one another.
\begin{prop}\label{variance : equivalence}
Let $(L_n)_n$ be a family of abelian extensions for which $\mathsf{GRH}$ holds and such that $d_{L_n} \longrightarrow\infty$. The following are equivalent:
\begin{enumerate}
    \item $(S_{L_n/\Q})_n$ converges uniformly to $0$.
    \item $\left(\Delta_{L_n/\Q}^{(2)}\right)_n$ converges uniformly to $\Gamma_2$.
    \item For all $r\geq 2$, $\left(\Delta_{L_n/\Q}^{(r)}\right)_n$ converges uniformly to $\Gamma_{r}$.
\end{enumerate}
\end{prop}
\begin{proof} We prove $(1)\Longrightarrow (3)\Longrightarrow (2) \Longrightarrow (1)$.
  Assume first that $\left(S_{L_n/\Q}\right)_n$ converges uniformly to $0$. By Corollary~\ref{negativity of U(a)}, we deduce that for all $n$ large enough we have \[\min_{a\in G_n\setminus\{1\}}(1-U_{L_n/\Q}(a)) \ge\frac{1}{2}\, .\]
  Fix a sufficiently large integer $n$ and let $a,b,c \in \mathcal{A}_3(G_n)$. 
    Write $U:=U_{L_n/\Q}$ and $S:=S_{L_n/\Q}$. By~\eqref{rhoii+1}, we have 
    \[\rho(t_{a,b}\, ,t_{b,c})=\frac{1}{2}\cdot\frac{-1+U(ab^{-1})+U(bc^{-1})-U(ac^{-1})}{\sqrt{(1-U(ab^{-1}))(1-U(bc^{-1}))}}\, .\]
    Since \[\bigl|\sqrt{1-U(bc^{-1})}-\sqrt{1-U(ab^{-1})}\bigr| \le |U(ab^{-1})-U(bc^{-1})|\, ,\]  
    we obtain
    \[ \left| \frac{-1+U(ab^{-1})}{\sqrt{(1-U(ab^{-1}))(1-U(bc^{-1}))}}+1\right|\le 4 \max_{\substack{x,y \in G_n\setminus\{1\} \\ x\ne y}} |S(x,y)|\, .\]
    Since \[ \left| \frac{U(bc^{-1})-U(ac^{-1})}{\sqrt{(1-U(ab^{-1}))(1-U(bc^{-1}))}}\right| \le 2\max_{\substack{x,y \in G_n\setminus\{1\} \\ x\ne y}} S(x,y)\] then,
    \[ \left|\rho(t_{a,b}\, ,t_{b,c})+\frac{1}{2}\right|\le 3\max_{\substack{x,y \in G_n\setminus\{1\} \\ x\ne y}} |S(x,y)| \, .\]
    If $a,b,c,d\in \mathcal{A}_4(G_n)$, by~\eqref{rhoij} we have 
    \[\rho(t_{a,b}\, ,t_{c,d})=\frac{1}{2}\cdot\frac{U(ab^{-1})+U(bd^{-1})-U(bc^{-1})-U(ad^{-1})}{\sqrt{(1-U(ab^{-1}))(1-U(cd^{-1}))}}\, .\]
    This clearly implies that \[\bigl|\rho(t_{a,b}\, ,t_{c,d})\bigr| \le 2\max_{\substack{x,y \in G_n\setminus\{1\} \\ x\ne y}} |S(x,y)| \, .\]
    We just proved that if $\left(S_{L_n/\Q}\right)_n$ converges uniformly to $0$, then for all $2\le r\le \liminf|G_n| $, we have that $\left(\Delta_{L_n/\Q}^{(r)}\right)_n$ converges uniformly to $\Gamma_r$. \\
    Since (3)$\Longrightarrow$(2) is immediate, we now prove (2)$\implies$(1) by contrapositive. 
    Assume that $\left(S_{L_n/\Q}\right)_n$ does not converge uniformly to $0$. Up to considering a subsequence, we may assume that there exist sequences $(a_n)_n$ and $(b_n)_n$ such that for all $n\ge 1$, $a_n,b_n\in G_n\setminus\{1\}$ and
    \[\lim_{n\to \infty}|U_{L_n/\Q}|(a_n)=\ell_1\, ,\quad \lim_{n\to \infty}|U_{L_n/\Q}|(b_n)=\ell_2\, , \text{ and } \lim_{n\to \infty}|U_{L_n/\Q}|\bigl(a_nb_n^{-1}\bigr)=\ell_3\, , \]
    with $\ell_1> \ell _2$. We have
    \[  \lim_{n\to \infty} \rho(t_{1,a_n}\, ,t_{a_n,b_n})=\frac{1}{2}\cdot \left(\frac{-1-\ell_1-\ell_3+\ell_2}{\sqrt{(1+\ell_1)(1+\ell_3)}}\right)\, .\]
    Thus, if $\ell_3> \ell_2$, then by symmetry of roles of $\ell_1$ and $\ell_3$, we may assume that $\ell_1\ge \ell_3$, and in this case we have
    \begin{align*}\lim_{n\to \infty} \rho(t_{1,a_n}\, ,t_{a_n,b_n})< \frac{1}{2}\cdot \frac{-1-\ell_1}{\sqrt{(1+\ell_1)(1+\ell_3)}}\le-\frac{1}{2}\, .\end{align*}
    If $\ell_3\le\ell_2$, then \begin{align*}\lim_{n\to \infty} \rho(t_{a_n,1}\, ,t_{1,b_n})&=\frac{1}{2}\cdot \left(\frac{-1-\ell_1-\ell_2+\ell_3}{\sqrt{(1+\ell_1)(1+\ell_2)}}\right)\\
    &\le   \frac{1}{2}\cdot \frac{-1-\ell_1}{\sqrt{(1+\ell_1)(1+\ell_2)}}<-\frac{1}{2}\, .\end{align*}
    Thus, $\left(\Delta_{L_n/\Q}^{(2)}\right)_n$ does not converge uniformly to $\Gamma_2$.

\end{proof}
In the preceding proof, we established a result that merits separate statement:
\begin{cor}\label{Nonconv}
    Let $(L_n)_n$ be a family of abelian extensions for which $\mathsf{GRH}$ holds and such that $d_{L_n} \longrightarrow\infty$. If $(S_{L_n/\Q})_n$ does not converge uniformly to $0$, then there exists an increasing sequence of integers $(n_m)_m$ and elements $x_m,y_m,z_m \in \mathcal{A}(G_{n_m})$ such that \[ \lim_{m\to \infty} \rho\bigl(t_{x_m,y_m},t_{y_m,z_m}\bigr)<-\frac{1}{2}\, .\]
\end{cor}
When the degrees tend to $\infty$, we have the following equivalence:
\begin{prop}\label{SeqsuivU}
    Let $(L_n)_n$ be a family of abelian extensions for which $\mathsf{GRH}$ holds and such that $[L_n:\Q] \to \infty$. Then $(S_{L_n/\Q})_n$ converges uniformly to $0$ if and only if $(U_{L_n/\Q})_n$ converges uniformly to $0$.
\end{prop}
\begin{proof}
    If $(U_{L_n/\Q})_n$ converges uniformly to $0$, then $(S_{L_n/\Q})_n$ clearly converges uniformly to $0$. Conversely, if $(S_{L_n/\Q})_n$ converges uniformly to $0$, we use Lemma~\ref{limU(x)}, which asserts that there exists $x_n\in G_n\setminus\{1\}$ such that $\displaystyle \lim_{n \to \infty} U_{L_n/\Q}(x_n)=0$. The result follows from the inequality:
    \[ \max_{a \in G_n\setminus\{1\} } \bigl|U_{L_n/\Q}(a)\bigr| \le \max_{\substack{x,y \in G_n\setminus\{1\} \\ x\ne y}} \bigl|S_{L_n/\Q}(x,y)\bigr|+|U_{L_n/\Q}(x_n)| \, . \]
\end{proof}

\subsection{Uniform moderacy}\label{subsec:uniformmod}

We begin this subsection with the proof of Theorem~\ref{thm : uniform-moderacy}:
\begin{proof}[Proof of Theorem~\ref{thm : uniform-moderacy}]
First note that, $(S_{L_n/\Q})_n$ converges uniformly to $0$ if and only if \[ \lim_{n\to \infty} \frac{1}{\log d_{L_n}}\max_{a,b\in G_n\setminus\{1\}} \Bigl| \sum_{\chi \in \Irr(G_n)}(\chi(a)-\chi(b)) \log A(\chi)\Bigr|=0\, .\] 
Assume that $(L_n)_n$ is uniformly $2$-moderate and that $(S_{L_n/\Q})_n$ converges uniformly to $0$.\\
By Proposition~\ref{variance : equivalence} we have 
\[ \max \left\{ \left \| \Delta(\bt) -\Gamma_{r-1} \right \|\, :\, \bt=(t_{a_1,a_2},\dots,t_{a_{r-1},a_r}),\ a_1,\dots a_r \in G_n\ \text{pairwise distinct }\right\}\underset{n\to \infty}{\longrightarrow} 0 \]
Thus, there exists $\beta>0$ such that for all sufficiently large $n$ and for all $a_1,\dots,a_r \in G_n$ pairwise distinct if $\bt=\left(t_{a_1,a_2},\dots,t_{a_{r-1},a_r}\right) $ we have $\lambda_{\bt}>\beta$. So that the main term of $\delta_{L_n/\Q}\left( \{a_1\},\dots,\{a_r\}\right)$ in the explicit formula (Theorem~\ref{explicit-formula}) is $O\left(\frac{1}{\sqrt{N_{L_n}}}\right)$ with implied absolute constants.
Moreover, the main term of $\delta_{L_n/\Q}\left( \{a_1\},\dots,\{a_r\}\right)$ is 
\[F_{r-1}\left( \left(B(t_{a_1,a_2}),\dots,B(t_{a_{r-1},a_r})\right);\Delta(\bt) \right)\, . \]
Since $$B(t_{a,b}) \ll \frac{r(G_n)}{\sqrt{\log d_{L_n}}}\quad (a\ne b \in G_n)\, ,$$
with implied absolute constants, it suffices to use the continuity of $F_{r-1}$ on $\R^{r-1}\times S_{r-1}^{++}$ to conclude that
$$ \max_{(a_1,\dots,a_r)\in \mathcal{A}_r(G_n)}\left|F_{r-1}\left( \left(B(t_{a_1,a_2}),\dots,B(t_{a_{r-1},a_r})\right);\Delta(\bt) \right)-F_{r-1}\left((0,\dots,0);\Gamma_{r-1} \right) \right|=o(1)
\, . $$
This proves that $(L_n)_n$ is uniformly $r$-moderate, in view of Lemma~\ref{cdf(0)}.\\
For the second implication, we proceed by contrapositive. Assume first that $r(G_n)/\sqrt{\log d_{L_n}}$ does not tend to $0$. Thus, $(L_n)_n$ is not $2$-moderate. By Lemma~\ref{r+1mod-rmod}, we deduce that $(L_n)_n$ is not $r$-moderate.
Assume now that $r(G_n)/\sqrt{\log d_{L_n}}\underset{n\to \infty}{\longrightarrow} 0$ and that $(S_{L_n/\Q})_n$ does not converge uniformly to $0$. By Corollary~\ref{Nonconv} and up to considering a subsequence, we may assume that there exist sequences $(x_n)_n,\, (y_n)_n,$ and $(z_n)_n$ such that for all $n\ge 1$ $x_n,y_n,z_n\in G_n$ and such that \[\lim_{n\to \infty} \rho\bigl(t_{x_n,y_n},t_{y_n,z_n}\bigr)<-1/2 \, .\]
By the three-way explicit formula (Corollary~\ref{3wayraces}) we deduce that $\displaystyle \lim_{n \to \infty} \delta_{L_n/\Q}^{(3)}(x_n,y_n,z_n)<1/6$, thus $(L_n)_n$ is not $3$-moderate. By Lemma~\ref{r+1mod-rmod}, we deduce that $(L_n)_n$ is not $r$-moderate. 
\end{proof}
We now deduce Corollary~\ref{primeorder-extensions}.
\begin{proof}[Proof of Corollary~\ref{primeorder-extensions}]
    By Lemma~\ref{negativity}, the value of the sum \[ \sum_{\chi \in \Irr(G_n)} \chi(a) \log A(\chi)\] does not depend on the choice of $a\in G_n\setminus\{1\}$. Thus, by Lemma~\ref{negativity of U(a)}, we have $(S_{L_n/\Q})_n$ converges uniformly to $0$. Hence the claim follows from Theorem~\ref{thm : uniform-moderacy}.
\end{proof}

\begin{exe}
    Let $(L_n)_n$ be a family of Galois extensions over $\Q$ with respective groups $G_n\simeq \Z/n\Z$ and suppose that for each $n\ge 1$ $L_n\subset \Q(\zeta_{p_n})$ with $p_n\equiv 1\pmod n$. Since $\Q(\zeta_{p_n})$ is totally ramified at $p_n$, it follows that $L_n$ is totally ramified at $p_n$. Thus, the inertia group of $L_n$ at $p_n$ is $G_n$. Note also that the ramification is tame, since $\gcd(p_n,n)=1$. By Lemma~\ref{negativity}, we deduce that for every $a\in G_n\setminus\{1\}$ we have \[\sum_{\chi\ne 1}\chi(a) \log A(\chi) =-\log p_n\, .\]
    In particular, by Lemma~\ref{negativity of U(a)}, we have
    \[U_{L_n/\Q}(a)=-\frac{1}{n-1}+o(1)=o(1)\quad (n\to \infty)\, .\]
    Noting that for all $n\ge 1$ we have $r(G_n)=1$, we conclude that $(L_n)$ is uniformly $r$-moderate for all $r\ge 2$.
\end{exe}
\begin{rk} \begin{enumerate}
    \item Note that Proposition~\ref{2mod-Udense} provides an example of a multiquadratic infinite tower that is uniformly $2$-moderate but not uniformly $r$-moderate for $r\ge 3$.
    \item Note also that when $|G_n|\to \infty$, then, by Proposition~\ref{SeqsuivU}, the uniform convergence to $0$ of $\bigl(S_{L_n/\Q}\bigr)_n$ in the characterization of $r$-moderacy can be replaced by the uniform convergence to $0$ of $\bigl(U_{L_n/\Q}\bigr)_n$.
\end{enumerate}
    
\end{rk}

We now give an example of a family $(L_n)_n$ that is uniformly $r$-moderate for all $r\ge 2$ but such that $\bigl(U_{L_n/\Q}\bigr)_n$ does not converge uniformly to $0$. 
\begin{exe}
    Let $q\ge 3$ be a prime number. Let $(p_n)_n$ be an increasing sequence of primes $p_n\equiv 1\mod q$. For each $n\ge 1$, let $L_n\subset \Q(\zeta_{p_n})$ be the subextension with group $G_n\simeq \Z/q\Z$. Since $L_n$ is totally ramified at $p_n$ (because $\Q(\zeta_{p_n})$ is totally ramified at $p_n$) and since the ramification is tame, we deduce that for all $a\in G_n\setminus\{1\}$ we have 
    \[ U_{L_n/\Q}(a)=-\frac{1}{q-1}+o(1)\quad (n\to \infty)\, .\]
    This proves that $\bigl(S_{L_n/\Q}\bigr)_n$ converges uniformly to $0$. Thus, $(L_n)_n$ is uniformly $r$-moderate for all $r\ge 2$, but $\bigl(U_{L_n/\Q}\bigr)_n$ converges uniformly to $-1/(q-1)\ne0$.
\end{exe}

\section{Density of values of $\delta_{L/\Q}^{(r)}$}\label{sec:density}
\begin{thm}\label{density}
        Let $(L_n)_n$ be a family of finite Galois extensions over $\Q$, for which $\mathsf{GRH}$ and $\mathsf{LI}$ hold, with respective Galois groups $(G_n)_n$, and such that $[L_n : \Q]\to \infty$. Assume that the function $$(a,b)\mapsto \frac{\| \widehat{t_{a,b}} \|_\infty^2}{V_{L_n/\Q}(t_{a,b})}$$
        converges uniformly to $0$. Then, $(B_{L_n/\Q})_n$ is dense in $\R$ if and only if $\left(\delta_{L_n/\Q}^{(2)}\right)_n$ is dense in $[0,1]$.
\end{thm}
\begin{proof}
This is a trivial consequence of the explicit formula (Corollary~\ref{r=1}), combined with the following two elementary facts: \begin{itemize}
    \item If $a_n=b_n+o(1)$ and $0\le a_n,b_n \le 1$ ($n\ge 1$), then $(a_n)_n$ is dense in $[\, 0\, ,1\,]$ if and only if $(b_n)_n$ is dense in $[\, 0\, ,1\,]$.
    \item If $f\colon \R\to (0,1)$ is a homeomorphism and $(B_n)_n$ is a real sequence, then $(B_n)_n$ is dense in $\R$ if and only if $(f(B_n))$ is dense in $[0,1]$.
\end{itemize}
\end{proof}

The following proposition provides a simplified abelian version of the previous theorem:
\begin{prop}
Let $(L_n)_n$ be an \emph{increasing} family of abelian extensions over $\Q$, for which $\mathsf{GRH}$ and $\mathsf{LI}$ hold, with respective Galois groups $(G_n)_n$, and such that $[L_n : \Q]\to \infty$. Let $a\in G_1$ be a non-square. For each $n\ge 1$, fix  a lift $a^{(n)}\in G_n$ of $a$. Then, $\left(r(G_n)/\sqrt{\log d_{L_n}}\right)_n$ is dense in $(0,\infty)$ if and only if $\left(\delta_{L_n/\Q}^{(2)}(a^{(n)},1)\right)_n$ is dense in $[0,1]$.
\end{prop}
\begin{proof}
    The proof is identical to that of Theorem~\ref{density}.
\end{proof}
\begin{exe}
Proposition~\ref{multiquadratic-B-density} provides an example of a family $(L_n)_n$ such that $\left(r(G_n)/\sqrt{\log d_{L_n}}\right)_n$ is dense in $(0,\infty)$. Hence, $\left(\delta_{L_n/\Q}^{(2)}(a^{(n)},1)\right)_n$ is dense in $(0,1)$ for any $a\in G_1$ that is non-square.
\end{exe}
In order to prove Theorem~\ref{main-density}, we need the following useful tool:
\begin{lem}[Slepian's Lemma]
Let $(X_1,\dots,X_r),(Y_1,\dots,Y_r)$ be centered Gaussian random vectors in $\R^r$ with respective cdfs $G_X$ and $G_Y$. If for all $1\leq i \leq r$ we have $\text{Var}(X_i)=\text{Var}(Y_i)$ and if for all $1\leq i,j \leq r$ we have $\text{Cov}(X_i,X_j)\leq \text{Cov}(Y_i,Y_j)$, then for all $\bx=(x_1,\dots,x_r)\in \R^r$:
$$G_X(x_1,\dots,x_r)\leq G_Y(x_1,\dots,x_r) \, .$$
\end{lem}
For $\rho \in \left[\, -\frac{1}{2}\, ,\, \frac{1}{2}\, \right]$ define the matrix $\Sigma_r \in \mathcal{M}_r(\R)$ by 
\begin{equation}\label{Sigmar}
\Sigma_r(\rho) = 
\begin{pmatrix}
1 & \rho & 0 & \cdots & 0 & 0 \\
\rho & 1 & -\frac{1}{2} & \ddots & & \vdots \\
0 & -\frac{1}{2} & 1 & \ddots & \ddots & 0 \\
\vdots & \ddots & \ddots & \ddots & -\frac{1}{2} & 0 \\
0 & & \ddots & -\frac{1}{2} & 1 & -\frac{1}{2} \\
0 & \cdots & 0 & 0 & -\frac{1}{2} & 1
\end{pmatrix}.
\end{equation}
It is easy to see that $\det \Sigma_1(\rho)=1,\ \det \Sigma_2(\rho)=1-\rho^2$ and that for $r\geq 3$ we have $$\det \Sigma_r(\rho)= \det \Sigma_{r-1}(\rho)-\frac{1}{4}\det \Sigma_{r-2}(\rho)\, .$$
This implies the general form $ \det \Sigma_r(\rho)=(r-2(r-1)\rho^2)/2^{r-1}$ for \(r\geq1\).
Thus, by Sylvester's criterion for all $\rho \in \left[\, -\frac{1}{\sqrt{2}}\, ,\, \frac{1}{\sqrt{2}}\, \right] $ the symmetric matrix $\Sigma_r(\rho)$ is positive-definite.\\
We now put Slepian's Lemma to use:
\begin{lem}\label{monotonicity}
For all $r\geq2$, the function 
\begin{align}\label{def : W_r}
W_r\colon\;[\, -1/\sqrt{2}\, ;\, 1/\sqrt{2}\, ]&\longrightarrow \R\\ \nonumber
\quad \rho\ \ \quad\ \  \ &\longmapsto F_r\bigl( (0,\dots,0); \Sigma_r(\rho) \bigr)
\end{align}
is strictly increasing.
\end{lem}
\begin{proof}
    The integrand in the expression of $F_r\bigl( (0,\dots,0); \Sigma_r(\rho) \bigr)$ is real-analytic with respect to $\rho$ on $ \left(\, -1/\sqrt{2}\, ;\, 1/\sqrt{2}\, \right)$. Thus, by standard results on parameter-dependent integrals of holomorphic functions, we deduce that $W_r$ is real-analytic on  $ \left(\, -1/\sqrt{2}\, ;\, 1/\sqrt{2}\, \right)$. Since $W_r(\rho)$ is the value at $(0,\dots,0)$ of the cdf of a centered Gaussian with covariance matrix $\Sigma_r(\rho)$, we use Slepian's Lemma to deduce that $W_r$ is non-decreasing. Thus, to prove that $W_r$ is strictly increasing it suffices to show that it is non-constant. By Lemma~\ref{cdf(0)} $W_r\left(-\frac{1}{2}\right)=\frac{1}{(r+1)!}$ and Applying Lemma~\ref{cdf(0)} again together with Fubini's Theorem, we deduce that \begin{align*}
         W_r(0)&= F_r\bigl((0,\dots,0);\Sigma_r(0)\bigr)=\frac{1}{2} F_{r-1}\bigl((0,\dots,0);\Gamma_{r-1} \bigr)\\
         &=\frac{1}{2\cdot r!}>\frac{1}{(r+1)!}=W_r\left(-\frac{1}{2}\right)\, .
    \end{align*}
\end{proof}
The following Lemma, a consequence of Lemma~\ref{limU(x)}, provides the control we need on the limit values of the covariance matrices.
\begin{lem}\label{main-matrices}
Let $(L_n)_n$ be a family of abelian extensions over $\Q$ for which $\mathsf{GRH}$ holds, with respective Galois groups $(G_n)_n$, such that $|G_n|\to \infty$. For each $n\geq1$, let $x_n \in G_n\setminus\{1\}$ such that $$\lim_{n\to \infty}\left|U_{L_n/\Q}(x_n)\right|=\ell\in (0,1]\, .$$
Fix $r\geq 2$, and fix $a_{n,1},\dots,a_{n,r}$ provided by the conclusion of Lemma~\ref{limU(x)}. We have the following:
\begin{equation}\label{uniform-conv-matrix}
    \lim_{n\to \infty} \Delta_{L_n/\Q}^{(r)}\left(x_n\, ,1\ ,a_{n,1}\ ,\dots,a_{n,r-1}\right)=\Sigma_r\left(-\frac{\sqrt{\ell+1}}{2}\right).
\end{equation}
\begin{equation}\label{extreme-val-matrix}
    \lim_{n\to \infty} \Delta_{L_n/\Q}^{(r)}\left(a_{n,1}x_n\, ,a_{n,2}x_n\, ,a_{n,1}\ ,\dots,a_{n,r-1}\right)=
    \begin{pmatrix}
           A & B \\[6pt]
           B^T & \Gamma_{r-2}
     \end{pmatrix},
\end{equation}
where \[
A 
= 
\begin{pmatrix}
1 & \dfrac{\ell - 1}{2} \\[8pt]
\dfrac{\ell - 1}{2} & 1
\end{pmatrix}
\;\in\;\mathcal{M}_{2}(\R)
\text{ and }
B 
= 
\begin{pmatrix}
-\ell      & \ell/2 & 0 & \cdots & 0 \\[8pt]
\dfrac{\ell - 1}{2}& -\ell/2 & 0 & \cdots & 0
\end{pmatrix}
\;\in\;\mathcal{M}_{2,\,r-2}(\R)\, .
\]
\end{lem}
\begin{proof}
Let $n\ge 1$. Denote $\bt_n=(t_{1,n},\dots,t_{r,n})$ where
\[ t_{1,n}=t_{x_n,1},\ t_{2,n}=t_{1,a_{n,1}},\ t_{i,n}=t_{a_{i-2,n},a_{i-1,n}}\quad (3\le i \le r)\, .\]
Since $\displaystyle \lim_{n \to \infty} |U_{L_n/\Q}(x_n)|=\ell>0$ and $U_{L_n/\Q}(x_n)=-|U_{L_n/\Q}(x_n)|+o(1)$, it follows that for all sufficiently large $n$, $U_{L_n/\Q}(x_n)<0$. Thus $\displaystyle \lim_{n \to \infty} U_{L_n/\Q}(x_n)=-\ell$. Since $U_{L_n/\Q}(a_{n,1}^{-1})=U_{L_n/\Q}(a_{n,1})$, it suffices to use Lemma~\ref{limU(x)} combined with~\eqref{rhoii+1} to deduce that 
\[ \lim_{n\to \infty}\rho( t_{1,n},t_{2,n} )=-\frac{\sqrt{1+\ell}}{2},\ \ \lim_{n\to \infty}\rho( t_{2,n},t_{3,n} )=-\frac{1}{2}\ \ \lim_{n\to \infty}\rho( t_{i+2,n},t_{i+3,n} )=-\frac{1}{2}\, .\]
Similarly, we apply Lemma~\ref{limU(x)} combined with~\eqref{rhoij} to deduce that \[ \lim_{n\to \infty} \rho(t_{1,n},t_{j+2,n})=\lim_{n\to \infty} \rho(t_{2,n},t_{j+3,n})=\lim_{n\to \infty}\rho(t_{i+2,n},t_{j+2,n})=0 \quad (j\ge 1, |i-j|\ge 2)\, .\]
This proves the first limit. The second one is also deduced from \eqref{rhoii+1} and \eqref{rhoij} and Lemma~\ref{limU(x)} in the same way.
\end{proof}

We are now ready to prove our main density result.
\begin{proof}[Proof of Theorem~\ref{main-density}]
Assume that the set of limit points of $(r(G_n)/\sqrt{\log d_{L_n}})_n$ has non-empty interior. Note that, since $(L_n)_n$ is increasing, then $(r(G_n))_n$ is an increasing sequence since the Galois groups are abelian. This implies that $r(G_n)\to \infty$. Let $n_0$ be sufficiently large, and let $a_2,\dots,a_r\in G_{n_0}$ non-squares. For all $n\ge n_0$, denote $a_1^{(n)}=1\in G_{n_0}$ and $a_{2}^{(n)},\dots,a_r^{(n)} \in G_n$ be respective lifts of $a_2,\dots,a_r$. Using the explicit formula (Theorem~\ref{explicit-formula}) combined with Proposition~\ref{matrix-conv}, we deduce
\[\delta_{L_n/\Q}^{(r)}\left( a_1^{(n)},\dots,a_r^{(n)} \right)=F_{r-1}\left(\bigl(-B\bigl(t_1^{(n)}\bigr),0,\dots,0\bigr);\Delta(\bt_n)\right)+O\left(\frac{1}{\sqrt{\log d_{L_n}}}\right) \, ,\]
where $\bt_n=\bigl(t_1^{(n)},\dots,t_{r-1}^{(n)}\bigr)$ and $t_i^{(n)}=t_{a_i^{(n)},a_{i+1}^{(n)}}$.
By Lemma~\ref{var-ab} we have \[ -B(t_1^{(n)})\underset{n\to \infty}{\sim} \frac{r(G_n)}{\sqrt{\log d_{L_n}}} \, .\]
Thus, the set of limit points of $ (-B(t_1^{(n)}))_n$ has non-empty interior. Moreover, by Proposition~\ref{matrix-conv} we have $\lim_{n\to \infty}\Delta(\bt_n)=\Gamma_{r-1}$. Let $z\in \R$ and $\mu>0$ such that all points of $[z-\mu,z+\mu]$ are limit points of $ (-B(t_1^{(n)}))_n $. By positivity of the integrand, we have 
\[  F_{r-1}\bigl((z-\mu,0,\dots,0);\Gamma_{r-1}\bigr)<F_{r-1}\bigl((z+\mu,0,\dots,0);\Gamma_{r-1}\bigr)\, .\]
Moreover, by continuity of $F$, for all $y\in [F_{r-1}\bigl((z-\mu,0,\dots,0);\Gamma_{r-1}\bigr),F_{r-1}\bigl((z+\mu,0,\dots,0);\Gamma_{r-1}\bigr)]$ there exist $x\in [z-\mu,z+\mu]$ such that 
\[F_{r-1}\bigl((x,0,\dots,0);\Gamma_{r-1}\bigr)=y\, ,   \]
Since $x$ is a limit point of $ (-B(t_1^{(n)}))_n $, then $y$ is a limit point of $\left(F\left(\bigl(-B\bigl(t_1^{(n)}\bigr),0,\dots,0\bigr);\Delta(\bt_n)\right)\right)_n$.
Thus, the set of limit points of the sequence $\left(F\left(\bigl(-B\bigl(t_1^{(n)}\bigr),0,\dots,0\bigr);\Delta(\bt_n)\right)\right)_n$ has non-empty interior, and since a $o(1)$ term does not change the set of limit points, thus the set of limit points of the sequence $\left( \delta_{L_n/\Q}^{(r)}\left( a_1^{(n)},\dots,a_r^{(n)} \right)  \right)_n$ has non-empty interior.\\
Assume that $r(G_n)/\sqrt{\log d_{L_n}} \longrightarrow 0$ and that there exists $0<\ell_1<\ell_2<1$ such that \[ [\, \ell_1\, , \ell_2\, ] \subset \overline{ \bigcup_{n\geq1} \left|U_{L_n/\Q}\right|(G_n\setminus\{1\})} \, .\]
Let $\ell \in [\, \ell_1\, , \ell_2\, ]$. Then, for all $n\ge 1$ there exists $x_n\in G_n$ such that
\[\lim_{n\to \infty} \left| U_{L_n/\Q}(x_n) \right|=\ell \, .\]
By Lemmas~\ref{limU(x)} and~\ref{main-matrices}, we deduce that there exists $a_{n,1},\dots,a_{n,r}$ such that
\[ \lim_{n\to \infty} \Delta_{L_n/\Q}^{(r-1)}(x_n,1,a_{n,1},\dots,a_{n,r-2} )=\Sigma_r\left(-\frac{\sqrt{\ell+1}}{2}\right) \, .\]
Thus we can take the limit in the explicit formula to obtain: 
\[\lim_{n\to \infty}\delta_{L_n/\Q}^{(r)}\left(x_n,1,a_{n,1},\dots,a_{n,r-2}\right)=F_{r-1}\left((0,\dots,0);\Sigma_r\bigl(-\sqrt{\ell+1}/2\bigr)\right)=W_r\left(-\frac{\sqrt{\ell+1}}{2}\right)\, .\]
By Lemma~\ref{monotonicity}, the function $W_r$ is strictly increasing. Thus, 
\[  \left[ W_r\bigl(-\sqrt{\ell_2+1}/2\bigr)\, ,\, W_r\bigl(-\sqrt{\ell_1+1}/2\bigr)\, \right] \subset \overline{\bigcup_{n\geq 1} \delta_{L_n/\Q}^{(r)}(\mathcal{A}_r(G_n))}\, , \]
which proves that $\overline{\bigcup_{n\geq 1} \delta_{L_n/\Q}^{(r)}(\mathcal{A}_r(G_n))} $ has non-empty interior.\\
Assume that $r(G_n)/\sqrt{\log d_{L_n}} \longrightarrow 0$ and that there exists $\ve>0$ such that $(\, 1-\ve,\, 1\, )\subset \overline{ \bigcup_{n\geq1} \left|U_{L_n/\Q}\right|(G_n\setminus\{1\})}$. Denote 
\begin{equation*}
   A_\ell=
    \begin{pmatrix}
           A & B \\[6pt]
           B^T & \Gamma_{r-3}
     \end{pmatrix},
\end{equation*}
where $A$ and $B$ are given by the statement of Lemma~\ref{main-matrices}.
Let $x_n\in G_n$, $n\ge 1$, such that \[\lim_{n \to \infty} |U_{L_n/\Q}(x_n)|=\ell \, ,\]
and let $a_{n,1},\dots,a_{n,r-2} \in G_n$ given by Lemma~\ref{limU(x)}. By Lemma~\ref{main-matrices} we deduce that 
\[\lim_{n\to \infty}\Delta_{L_n/\Q}^{(r)}\left(a_{n,1}x_n\, ,a_{n,2}x_n\, ,a_{n,1}\ ,\dots,a_{n,r-2}\right)=A_\ell \, .\]
Using the fact that $A_\ell \in S_{r-1}^{++}$ (that is because $\ell \in (0,1)$, it can be proved by checking that the Schur complement of the matrix $\Gamma_{r-3}$ in $A_\ell$ is positive definite). The explicit formula implies that 
\[ \lim_{n\to \infty} \delta_{L_n/\Q}^{(r)}\bigl(a_{n,1}x_n;a_{n,2}x_n,a_{n,1},\dots,a_{n,r-2} \bigr)=F_{r-1}\left((0,\dots,0);A_\ell\right) \, .\]
Let $\bX=\left(X_1,\dots,X_{r-1} \right)$ be a random centered Gaussian with covariance matrix $A_\ell$. We have \[ P\bigl(X_1\le 0,\dots,X_{r-1}\le 0\bigr) \le P(X_1\le 0,X_3\le 0)\, .\]
Thus, \begin{equation*}  \lim_{n\to \infty} \delta_{L_n/\Q}^{(r)}\bigl(a_{n,1}x_n;a_{n,2}x_n,a_{n,1},\dots,a_{n,r-2} \bigr) \le \frac{1}{4}-\frac{1}{2\pi}\arcsin(\ell)\, , \end{equation*}
which proves that when $\ell$ is sufficiently close to $1$, $  \lim_{n\to \infty} \delta_{L_n/\Q}^{(r)}\bigl(a_{n,1}x_n;a_{n,2}x_n,a_{n,1},\dots,a_{n,r-2} \bigr)$ is sufficiently close to $0$.
\end{proof}
\begin{proof}[Proof of Theorems~C and D] The bound of $\delta_{L_n/\Q}^{(3)}$ in Theorem~D is a consequence of Corollary~\ref{abelian-3wayraces} combined with Lemma~\ref{encadrementrho}. The rest of the statements are consequences of
Proposition~\ref{2mod-Udense} which establishes the existence a sequence of primes whose corresponding multiquadratic extensions satisfy the assumptions of Theorem~\ref{main-density}.

\end{proof}
\begin{rk}
Using the same ideas, one can recover Fiorilli's result~\cite{Fiorilli}*{Theorem 1.1} which assumes $\mathsf{GRH}$ and $\mathsf{LI}$. Indeed, let $p_1,\dots,p_n$ be an increasing sequence of primes congruent to $1\pmod 4$, set $L_n=\Q(\sqrt{p_1},\dots,\sqrt{p_n})\subset \Q\bigl(\zeta_{p_1\dots p_n}\bigr)=L_n^+$ (the last inclusion is due to the classical fact that $p_i$ is a square in $\Q\bigl(\zeta_{p_i}\bigr)$). Using properties of Artin $L$-functions in towers, one easily shows that, if $\pi_n\colon\Gal(L_n^+/\Q)\longrightarrow\Gal(L_n/\Q)$ is the natural projection, then for any class function $t:G_n\rightarrow\C$ we have \begin{equation}\label{Artin-tower}
    \pi(x;L_n/\Q;t)=\pi(x;L_n^+/\Q;t\circ \pi_n)\, .
\end{equation}
Now let $$t_n=|G_n|\left(\mathds{1}_{\{1\}}-\frac{1}{2^n-1}\sum_{a\in G_n\setminus\{1\} }\mathds{1}_{\{a\}}\right)\, .$$
Thus, $\mathcal{P}_{t_n}=\mathcal{P}_{t\circ \pi_n}$. Noting that $\pi_n^{-1}\bigl(\{1\}\bigr)$ is exactly the set of squares of $G_n^+$, we deduce that the logarithmic density of $\mathcal{P}_{t_n\circ \pi_n}$, is the same as that studied by Fiorilli for races between quadratic residues and non-residues. Thus, in order to study this logarithmic density, it suffices to study the logarithmic density of $\mathcal{P}_{t_n}$, with the same arguments as above one can show that $B_{L_n/\Q}(t_n)$ is dense in $\R_{<0}$ which leads to the density of logarithmic densities in $[1/2,1]$.
\end{rk}

\section*{Appendix}
\begin{proof}[Proof of Lemma~\ref{PA}]
    First notice that for all $i\in \{1,\dots,r\}$ we have 
    \[
    \frac{\partial |f|^{2n}}{\partial x_i}=(a_i \overline{f}+\overline{a_i}f)|f|^{2(n-1)}\, ,
    \]
    and if we have 
    \[
    \frac{\partial^\alpha |f|^{2n}}{\prod_{i\in A}\partial x_i}(\ba,\bx)=P_A(\ba,\overline{\ba},\bx) |f|^{2(n-\alpha)}(\ba,\bx)\, ,
    \]
    then, for $j\in L\setminus A$, denoting $B=A\cup \{j\}$, we have
     \[
    \frac{\partial^{\alpha+1} |f|^{2n}}{\prod_{i\in B}\partial x_i}(\ba,\bx)=P_B(\ba,\overline{\ba},\bx) |f|^{2(n-\alpha-1)}(\ba,\bx)\, ,
    \]
    where 
    \begin{equation}
        P_B(\ba,\overline{\ba},\bx)=\frac{\partial P_A}{\partial X_j}(\ba,\overline{\ba},\bx)f(\ba,\bx)f(\overline{\ba},\bx)+(n-\alpha)P_A(\ba,\overline{\ba},\bx)(a_j f(\overline{\ba},\bx)+\overline{a_j}f(\ba,\bx))\, .
    \end{equation}
    Thus, we can write: 
    \begin{equation}\label{PB}
P_B(\bS,\bT,\bX) = \frac{\partial P_A}{\partial X_j}(\bS,\bT,\bX)\cdot (\bS \cdot \bX)(\bT \cdot \bX)
+ (n-\alpha) P_A(\bS,\bT,\bX) \cdot \left(S_j (\bT \cdot \bX) + T_j (\bS \cdot \bX)\right)\, ,
\end{equation}
where the inner product $\cdot$ is defined by $\bS \cdot \bX:=\sum_{i=1}^rS_iX_i$. The existence is thus proved by induction. We now prove the properties when $n\le \alpha$. Note that when $A=\emptyset$ we have $P_A=1$ which satisfies all the properties stated. We assume by induction that each property holds for $\alpha<n$ and let $j\in L\setminus A$ and denote $B=A\cup \{j\}$, we want to prove that each property hold for $B$: 
(1) The equation \eqref{PB} shows that $P_B$ is the sum of products of homogeneous polynomials in each tuple of indeterminates $\bS,\, \bT$ and $\bX$ (this is thanks to the induction hypothesis). Moreover, since all coefficients of $P_A$ are non-negative integers, then so are those of $P_B$. This completes the induction for the first property when $n\le \alpha$.\\
(2) Using the fact that $P_B$ is homogeneous in each tuple of indeterminates, it is enough to consider one of its monomials, and since $P_A$ has non-negative coefficients it is in fact enough to compute the degree of $S_jT_1X_1P_A(\bS,\bT,\bX)$ at each tuple of indeterminates. Using the induction hypothesis we obtain the result.\\
(3) By equation \eqref{PB} the coefficient of a monomial of $P_B$ is $\le \alpha c_A+(n-\alpha)c_A=nc_A$, where $c_A$ is the coefficient of some monomial of $P_A$. Thus, by the induction hypothesis the coefficient of a monomial of $P_B$ is $\le n^{\alpha+1}$.\\
This completes the proof of each property when $\alpha \le n$. Let $A\subset \{\, 1,\dots,r\}$ of cardinality $\alpha$, where $n<\alpha\leq 2n$, and let $C\subset A$ with cardinality $n$. Thus, 
\[
P_A(\bS,\bT,\bX)=\frac{\partial^{\alpha-n}P_C}{\prod_{i\in A\setminus C}X_i}(\bS, \bT,\bX).
\]
All the properties, for $n<\alpha \le 2n$, follow easily from the previous equality in addition to the fact that $P_C$ satisfies the previous properties.
\end{proof}

\noindent\textbf{Acknowledgments.} We warmly thank Florent Jouve for introducing us to the subject and for his valuable suggestions and remarks that helped improve the paper, as well as Daniel Fiorilli for several insightful discussions that contributed to the development of our ideas.


\begin{bibdiv}
\begin{biblist}
    \bib{Bail}{article}{
   author={Bailleul, Alexandre},
   title={Chebyshev's bias in dihedral and generalized quaternion Galois
   groups},
   journal={Algebra Number Theory},
   volume={15},
   date={2021},
   number={4},
   pages={999--1041},
}




\bib{Dev}{article}{
   author={Devin, Lucile},
   title={Chebyshev's bias for analytic $L$-functions},
   journal={Math. Proc. Cambridge Philos. Soc.},
   volume={169},
   date={2020},
   number={1},
   pages={103--140},
}

\bib{Fiorilli}{article}{
   author={Fiorilli, Daniel},
   title={Highly biased prime number races},
   journal={Algebra Number Theory},
   volume={8},
   date={2014},
   number={7},
   pages={1733--1767},
   issn={1937-0652},
   review={\MR{3272280}},
   doi={10.2140/ant.2014.8.1733},
}


\bib{FJ}{article}{
   author={Fiorilli, Daniel},
   author={Jouve, Florent},
   title={Distribution of Frobenius elements in families of Galois
   extensions},
   journal={J. Inst. Math. Jussieu},
   volume={23},
   date={2024},
   number={3},
   pages={1169--1258},
}

\bib{FJ2}{article}{
   author={Fiorilli, Daniel},
   author={Jouve, Florent},
   title={Unconditional Chebyshev biases in number fields},
   journal={J. \'{E}c. polytech. Math.},
   volume={9},
   date={2022},
   pages={671--679},
}



\bib{FiM}{article}{
   author={Fiorilli, Daniel},
   author={Martin, Greg},
   title={Inequities in the Shanks-R\'enyi prime number race: an asymptotic
   formula for the densities},
   journal={J. Reine Angew. Math.},
   volume={676},
   date={2013},
   pages={121--212},
   issn={0075-4102},
   review={\MR{3028758}},
   doi={10.1515/crelle.2012.004},
}

\bib{HL}{article}{
   author={Harper, Adam J.},
   author={Lamzouri, Youness},
   title={Orderings of weakly correlated random variables, and prime number
   races with many contestants},
   journal={Probab. Theory Related Fields},
   volume={170},
   date={2018},
   number={3-4},
   pages={961--1010},
   issn={0178-8051},
   review={\MR{3773805}},
   doi={10.1007/s00440-017-0800-2},
}

\bib{hayani}{article}{
  title        = {On the influence of the Galois group structure on the Chebyshev bias in number fields},
  author       = {Hayani, Mounir},
  note         = {forthcoming, Bulletin de la SMF, available at \url{https://inria.hal.science/hal-04538383}},
  year         = {2024},
  month        = {4},
  keywords     = {Number theory, Chebyshev's bias},
  url          = {https://inria.hal.science/hal-04538383},
  pdf          = {https://inria.hal.science/hal-04538383/file/Chebyshev-bias-prep.pdf},
}

\bib{HK}{article}{
   author={Heuberger, Clemens},
   author={Kropf, Sara},
   title={Higher dimensional quasi-power theorem and Berry-Esseen
   inequality},
   journal={Monatsh. Math.},
   volume={187},
   date={2018},
   number={2},
   pages={293--314},
   issn={0026-9255},
   review={\MR{3850313}},
   doi={10.1007/s00605-018-1215-6},
}













\bib{Kotz}{book}{
   author={Kotz, Samuel},
   author={Balakrishnan, N.},
   author={Johnson, Norman L.},
   title={Continuous multivariate distributions. Vol. 1},
   series={Wiley Series in Probability and Statistics: Applied Probability
   and Statistics},
   edition={2},
   note={Models and applications},
   publisher={Wiley-Interscience, New York},
   date={2000},
   pages={xxii+722},
   isbn={0-471-18387-3},
   review={\MR{1788152}},
   doi={10.1002/0471722065},
}

\bib{LO}{article}{
   author={Lagarias, J. C.},
   author={Odlyzko, A. M.},
   title={Effective versions of the Chebotarev density theorem},
   conference={
      title={Algebraic number fields: $L$-functions and Galois properties},
      address={Proc. Sympos., Univ. Durham, Durham},
      date={1975},
   },
   book={
      publisher={Academic Press, London-New York},
   },
   date={1977},
   pages={409--464},
   review={\MR{0447191}},
}


\bib{Lam}{article}{
   author={Lamzouri, Youness},
   title={Prime number races with three or more competitors},
   journal={Math. Ann.},
   volume={356},
   date={2013},
   number={3},
   pages={1117--1162},
   issn={0025-5831},
   review={\MR{3063909}},
   doi={10.1007/s00208-012-0874-1},
}


\bib{Leshin}{article}{
   author={Leshin, Jonah},
   title={Solvable number field extensions of bounded root discriminant},
   journal={Proc. Amer. Math. Soc.},
   volume={141},
   date={2013},
   number={10},
   pages={3341--3352},
   issn={0002-9939},
   review={\MR{3080157}},
   doi={10.1090/S0002-9939-2013-12015-0},
}

\bib{LM}{article}{
   author={Lin, Jiawei},
   author={Martin, Greg},
   title={Densities in certain three-way prime number races},
   journal={Canad. J. Math.},
   volume={74},
   date={2022},
   number={1},
   pages={232--265},
   issn={0008-414X},
   review={\MR{4379402}},
   doi={10.4153/S0008414X20000747},
}

\bib{Littlewood}{article}{
   author={Littlewood, J. E.},
   title={On the Class-Number of the Corpus $P({\surd}-k)$},
   journal={Proc. London Math. Soc. (2)},
   volume={27},
   date={1928},
   number={5},
   pages={358--372},
   issn={0024-6115},
   review={\MR{1575396}},
   doi={10.1112/plms/s2-27.1.358},
}


\bib{Ng}{article}{
   author={Ng, Nathan},
   title={Limiting distributions and zeros of Artin L-functions},
   type={PhD Thesis},
   year={2000},
   eprint={https://www.cs.uleth.ca/~nathanng/RESEARCH/phd.thesis.pdf},
   publisher= {University of British Columbia},
}

\bib{Ost}{article}{
  title={{\"U}ber dirichletsche reihen und algebraische differentialgleichungen},
  author={Ostrowski, Alexander},
  journal={Mathematische Zeitschrift},
  volume={8},
  number={3},
  pages={241--298},
  year={1920},
  publisher={Springer}
}


\bib{RS94}{article}{
   author={Rubinstein, Michael},
   author={Sarnak, Peter},
   title={Chebyshev's bias},
   journal={Experiment. Math.},
   volume={3},
   date={1994},
   number={3},
   pages={173--197},
}



    
\end{biblist}
\end{bibdiv}
\end{document}